\documentclass[sn-mathphys]{sn-jnl}

\usepackage{amsmath}
\usepackage{amssymb}
\usepackage{amsthm}
\usepackage{color}
\usepackage{float}
\usepackage{geometry}

\theoremstyle{thmstyleone}%

\theoremstyle{plain}\newtheorem{theorem}{Theorem}[section]
\theoremstyle{definition}\newtheorem{Defi}[theorem]{Definition}\theoremstyle{plain}\newtheorem{lemma}[theorem]{Lemma}
\newtheorem{coro}[theorem]{Corollary}
\newtheorem{proposition}{Proposition}[section]

\usepackage{epstopdf}
\usepackage{subfigure}

\begin{document}

\title[Convergence of Hamiltonian Particle methods]{Convergence of Hamiltonian Particle methods\\ for Vlasov--Poisson equations with a nonhomogeneous magnetic field}


\author[1]{\fnm{Anjiao} \sur{Gu}}\email{gaj2023@sjtu.edu.cn}

\author*[2,3]{\fnm{Yajuan} \sur{Sun}}\email{sunyj@lsec.cc.ac.cn}


\affil[1]{\orgdiv{School of Mathematical Sciences}, \orgname{ Shanghai Jiao Tong University}, \city{Shanghai}, \postcode{200240}, \country{China}}

\affil[2]{\orgdiv{LSEC, ICMSEC, Academy of Mathematics and Systems Science}, \orgname{Chinese Academy of Sciences}, \city{Beijing}, \postcode{100190}, \country{China}}

\affil[3]{\orgdiv{School of Mathematical Sciences}, \orgname{University of Chinese Academy of Sciences}, \city{Beijing}, \postcode{100049}, \country{China}}



\abstract{
In high-temperature plasma physics, a strong magnetic field is usually used to confine charged particles. 
Therefore, for studying the classical mathematical models of the
physical problems it is needed to consider the effect of external magnetic fields. 
One of the important model equations in plasma is the Vlasov--Poisson equation with an external magnetic field.
In this paper, we study the error analysis of Hamiltonian particle methods for this kind of system. 
The convergence of particle method for Vlasov equation and that of Hamiltonian method for particle equation are provided independently. 
By combining them, it can be concluded that the numerical solutions converge to the exact particle trajectories.
}




\maketitle

\section{Introduction}\label{sec1}

Particle methods have become a significant numerical simulation tool in plasma physics recently.
In these methods, the solutions are sought as macro particles with the weighted Klimontovich representation. The trajectories of macro particles are obtained by solving the characteristic equation.
They can reproduce well the realistic physical
phenomena~\cite{Hockney1988,Birdsall1991} and have a relatively small computational cost to conduct high dimensional simulations~\cite{Nicolas2009,LYZ2019}.
Due to these factors, they have been widely used for approximating solutions of PDEs such as Vlasov type in plasma physics.
In the framework of geometric integral, the Vlasov--Maxwell system can be written as a Hamiltonian system with respect to the Morrison-Marsden-Weinstein (MMW) Poisson bracket~\cite{Morrison1980T,Marsden1982T}.
The Vlasov--Poisson system with a nonhomogeneous magnetic field also has the Poisson bracket structure and has very closed relation to the MMW Poisson bracket~\cite{GAJ2022Hamiltonian}.
This equation is important for simulating the dynamics of tokamak plasmas.
For this type of equation, the magnetic field is applied to confine a plasma constituted of a large number of charged particles.
According to the Poisson bracket of the magnetized Vlasov--Poisson equations, we design the numerical algorithms by following the idea called geometric numerical integrators which can conserve the intrinsic properties of the original system~\cite{FK1985,Hairer2006,FK2010}.
Certainly, structure-preserving particle methods have been widely developed for Vlasov equations~\cite{Squire2012G,XJY2015E,Qin2016C,HY2016,Kraus2017,LYZ2019,Martin2021Variational}.

There are already some works for the convergence analysis of particle methods on unmagnetized Vlasov--Poisson system.
For one-dimensional Vlasov--Poisson systems, research~\cite{Cottet1984Particle} gave a mathematical analysis that showed the convergence of the Hamiltonian trajectories of the mollified equations which means it is mollified by convolving with a smooth function to those of the unmollified problem.
This results can be generally extended to the multidimensional cases~\cite{Victory1989On,Victory1991TheS,Victory1991TheF}.
On the other hand, there are already some particle methods constructed for Vlasov--Poisson system with a strong magnetic field~\cite{Filbet2016,Chartier2020}.
However, error analysis of them is lacking.
Thus, an appropriate tool for the error estimation is urgent.
In \cite{GAJ2022Hamiltonian}, we have developed the Hamiltonian Particle-in-Cell methods for magnetized Vlasov--Poisson system. Our aim in this paper is to build a complete error analysis of it.

We study the errors incurred for the fully discretized systems of the magnetized Vlasov--Poisson equation under the maximal ordering scaling.
For the convergence analysis of particle methods, we compare the electric field between the mollified equation and the original one.
On the other hand, we applied the averaging technique for splitting scheme studied in~\cite{ZXF2021Error} to get an uniform error bound for our Poisson bracket preserving splitting method.
Moreover, the linear part of the splitting method is exact so that the result is only depend on the time step.
Eventually, we can come to a conclusion from that the numerical solutions of Hamiltonian Particle methods converge to the original system.

The outline of the paper is as follows. In section 2, we introduce the the Vlasov equation with an external magnetic field.
For particle method, in section 3 we present the description and convergence analysis of it.
Furthermore, we analyze the Hamiltonian Particle-in-Cell method in Section 4.
Error analysis of the fully discretization method is presented and a refined one of the parallel component of the velocity is discussed further.
The main part is the optimal convergence of temporal discretization owing to the spatial discretization is a standard conforming finite element method.
Finally, we conclude this paper.

\section{Vlasov--Poisson equations}

The Vlasov--Poisson system with an external magnetic field is given by
\begin{align}
&\frac{\partial f}{\partial t}+{v}\cdot\frac{\partial f}{\partial {x}}+(E+v\times B)\cdot\frac{\partial f}{\partial {v}}=0,\label{VP01}\\
&E=-\nabla\phi,\label{VP02}\\
&-\Delta\phi= {\rho(x,t)-\rho_0},\label{VP03}\\
&\rho=\int_{\mathbb{R}^3}f dv,\label{VP04}
\end{align}
where $x,v\in \mathbb{R}^3$ and $f(x,v,0)=f_0(x,v)$.
Here, $B$ is an external magnetic field, $\rho$ is charge density function and  $\rho_0$ is a constant which represents ion density. Also, the Vlasov equation (\ref{VP01}) can be written in the following divergence form
\begin{equation*}
\frac{\partial f}{\partial t}+\nabla_{x}\cdot(vf)+\nabla_{v}\cdot((E+v\times B)f)=0.
\end{equation*}
Equations (\ref{VP03})-(\ref{VP04}) can be solved via the fundamental solution, that is
\begin{equation}\label{phi}
\phi(x)=\frac{1}{4\pi}\int_{\mathbb{R}^{3}} \frac{1}{\lvert x-y \rvert}(\int_{\mathbb{R}^{3}}f(y,v,t)dv-\rho_0)dy
\end{equation}
Substituting (\ref{phi}) into (\ref{VP02}) gives
\begin{equation}\label{E0}
E=\int_{\mathbb{R}^{3}} K(x,y)(\int_{\mathbb{R}^{3}}f(y,v,t)dv-\rho_0)dy,
\end{equation}
where $K(x,y)=\frac{1}{4\pi}\frac{x-y}{\lvert x-y \rvert^3}$.

In~\cite{GAJ2022Hamiltonian}, we have presented the following Poisson bracket of the system
\begin{equation}
\begin{aligned}
\{\{\mathcal{F},\mathcal{G}\}\}  (f)
&=\int_{\mathbb{R}^3}\int_{\mathbb{R}^3} f\left\{ \frac{\delta\mathcal{F}}{\delta f},\frac{\delta\mathcal{G}}{\delta f}\right\}_{xv}dxdv\\
&+\int_{\mathbb{R}^3}\int_{\mathbb{R}^3} fB\cdot\left(\frac{\partial}{\partial v}\frac{\delta\mathcal{F}}{\delta f}\times\frac{\partial}{\partial v}\frac{\delta\mathcal{G}}{\delta f}\right)dxdv,
\end{aligned}
\label{eq:MMWB}
\end{equation}
where $\mathcal{F}$ and $\mathcal{G}$ are two functionals of $f$, $\frac{\delta \mathcal{F}}{\delta f}$ is the variational derivative. The operator $\left\{ \cdot,\cdot\right\} _{xv}$ is the canonical Poisson bracket which for two given functions $m(x,v)$ and  $n(x,v)$, is
\[
\{m,n\}_{xv}=\frac{\partial m}{\partial x}\cdot\frac{\partial n}{\partial v}-\frac{\partial m}{\partial v}\cdot\frac{\partial n}{\partial x}.
\]

With the bracket (\ref{eq:MMWB}), we consider the following Poisson system
\begin{equation}
\frac{d\mathcal{F}}{d t}=\{\{\mathcal{F},\mathcal{H}\}\},\label{eq:PoissonVM}
\end{equation}
where $\mathcal{F}$ is any functional of $f$, and $\mathcal{H}$ is the Hamiltonian functional which is also the global energy of the system,
\begin{equation}
\begin{aligned}
\mathcal{H}[f]
& =\frac{1}{2}\int_{\mathbb{R}^3}\int_{\mathbb{R}^3}v^{2}fdxdv+\frac{1}{2}\int_{\mathbb{R}^3}E^{2}dx\\
& =\frac{1}{2}\int_{\mathbb{R}^3}\int_{\mathbb{R}^3}v^{2}fdxdv+\frac{1}{2}\int_{\mathbb{R}^3}\int_{\mathbb{R}^3}\phi fdxdv-\frac{1}{2}\rho_0\int_{\mathbb{R}^3}\phi dx.\label{eq:HamiltonVM}
\end{aligned}
\end{equation}
In fact, by setting $\mathcal{F}[f]=\int_{\mathbb{R}^3}\int_{\mathbb{R}^3}f(\tilde{x},\tilde{v},t)\delta(x-\tilde{x})\delta(v-\tilde{v})d\tilde{x}d\tilde{v}$, and defining the local energy $h(x,v)=\frac{\delta\mathcal{H}}{\delta f}(x,v)={v^{2}}/2+\phi({x})$, it follows from (\ref{eq:PoissonVM}) that
\begin{align}\label{eq:poisson3}
\frac{\partial f}{\partial t}&=-\{f,h\}_{xv}-B\cdot\left(\frac{\partial f}{\partial v}\times \frac{\partial h}{\partial v}\right)
\end{align}
which recovers the Vlasov equation (\ref{VP01}).  This implies that
the magnetized Vlasov--Poisson equation can be written in a Poisson system.

\begin{Defi}
A function $f$ is called a weak solution of the Cauchy problem (\ref{VP01})-(\ref{VP04}) if
\begin{equation}\label{weakform}
\begin{aligned}
&\int_0^T\int_{\mathbb{R}^3}\int_{\mathbb{R}^3} f(x,v,t)(\psi_t(x,v,t)+v\cdot\nabla_x\psi(x,v,t)+(E+v\times B)\cdot\nabla_v\psi(x,v,t))dxdvdt\\
&+\int_{\mathbb{R}^3}\int_{\mathbb{R}^3} f_0\psi(x,v,0)dxdv=0
\end{aligned}
\end{equation}
holds for any test function $\psi\in C_0^{\infty}([0,T)\times\mathbb{R}^3\times\mathbb{R}^3)$.
\end{Defi}

In~\cite{Rege2021Propagation}, for the magnetized Vlasov--Poisson systems (\ref{VP01})-(\ref{VP04}) the existence and uniqueness of the solution has been proved. We list some known results as follows.
\begin{theorem}\label{VPexistence}
Assume that the magnetic field $B\in W^{1,\infty}(\mathbb{R}^{3})$, the initial function $f_0$ belongs to the space $L^1(\mathbb{R}^{3}\times\mathbb{R}^{3})\cap L^{\infty}(\mathbb{R}^{3}\times\mathbb{R}^{3})$ and satisfies
\begin{equation*}
\begin{aligned}
\int_{\mathbb{R}^3}\int_{\mathbb{R}^3}\lvert v \rvert^{\gamma}f_0 dxdv<+\infty
\end{aligned}
\end{equation*}
for some $\gamma>2$.
Then, there exists a weak solution $f$ of  system (\ref{VP01})-(\ref{VP04}) for a finite interval $[0,T]$.
\end{theorem}

To derive the uniqueness result of the solution, the extra conditions of $f_0$ is needed. It is shown in the following theorem.
\begin{theorem}\label{VPuniqueness}
Let $B\in W^{1,\infty}(\mathbb{R}^{3})$ and $f_0\in L^1(\mathbb{R}^{3}\times\mathbb{R}^{3})\cap L^{\infty}(\mathbb{R}^{3}\times\mathbb{R}^{3})$. Assume that
\begin{equation*}
\int_{\mathbb{R}^3}\int_{\mathbb{R}^3}\lvert v \rvert^{6}f_0 dxdv<+\infty\quad and\quad \int_{\mathbb{R}^3}\int_{\mathbb{R}^3}\lvert x \rvert^{4}f_0 dxdv<+\infty
\end{equation*}
and the density $\rho$ satisfies
\begin{equation*}
\Vert\rho\Vert_{L^{\infty}([0,T]\times\mathbb{R}^{3})}<+\infty,
\end{equation*}
then the weak solution $f$ of system (\ref{VP01})-(\ref{VP04}) is unique.
\end{theorem}

In other sides, the equation (\ref{VP01}) can be considered as a hyperbolic system, where with $B$ is given, and $E$ satisfies the equations (\ref{VP02})-(\ref{VP04}). For the initial value $f(x,v,t=0)=f_0(x,v)$, under some smoothness assumptions for the solution $f$ we can use the characteristic line approach to solve the solution of system  (\ref{VP01}).
Consider the following characteristic equation
\begin{equation}\label{VP1}
\begin{aligned}
&\frac{dX}{dt}=V,\\
&\frac{dV}{dt}=E(X)+V\times B(X),\\
&X(0)=x,V(0)=v.
\end{aligned}
\end{equation}
We denote $(X(t;x,v,0),V(t;x,v,0))$ the solution of (\ref{VP1}).
Then the solution of the system can be expressed as
\begin{equation}\label{sol1}
f(x,v,t)=f_0(X(0;x,v,t),V(0;x,v,t)).
\end{equation}
Moreover, the flow is volume-preserving so that
\begin{equation}\label{sol2}
\frac{\partial(X,V)}{\partial(x,v)}=1.
\end{equation}

\section{Particle method of Vlasov--Poisson equation}
In numerical framework, the idea of Particle-in-Cell approach is to use two kinds of grids: the Euler grid and the Lagrange grid. Because of that the distribution function is represented by 'super particles' in the form of the weighted Klimontovich expression, particle method is to approximate and to follow the trajectories of super particles. This numerical computation can simulate well the realistic physical phenomena, and has been considered as an effective way to simulate plasmas in the kinetic theory.
In this section, we study the error analysis in the framework of the Particle method for magnetized Vlasov--Poisson equation. Herein, the Poisson equation is solved by using a mollified charge density at each time step. This implies that the approximate field quantities which is denoted by $E^h$ can be obtained by using a mollified Poisson kernel $K_r$ convolved with the density $\rho$.
\subsection{Description of the method}
We start this section from considering the following approximation equation defined on $\Omega\subseteq\mathbb{R}^{6}$
\begin{equation}\label{VPapp}
\frac{\partial f^h}{\partial t}+{v}\cdot\frac{\partial f^h}{\partial {x}}+(E^h+v\times B)\cdot\frac{\partial f^h}{\partial {v}}=0,\quad f^h(t=0)=f_0^h,
\end{equation}
where $E^h$ is an approximation of the electrostatic field $E$.

As follows, we seek the solution that can be expressed in a Dirac sum $$f^h(x,v,t)=\sum\limits_{j=1}^N\alpha_j\delta(x-X_j^h(t))\delta(v-V_j^h(t)),$$ where $X_j^h,V_j^h\in {\mathbb{R}^3},j=1\cdots N$ which obey some ODEs shown later.
We divide $\Omega$ to $N$ nonoverlapping subdomains $\Omega_j$ such that the initial position of $j$th particle $(X_j^h(0),V_j^h(0))$ is located at the center of $\Omega_j$. Then the initial discrete distribution function can be defined by 
\begin{equation}\label{f0htop}
f_0^h=\sum\limits_{j=1}^{N}\alpha_j\delta(x-X_j^h(0))\delta(v-V_j^h(0)).
\end{equation}
The weights $\alpha_j$ are chosen such that for any function
$\psi\in C_0^{\infty}([0,T)\times\mathbb{R}^3\times\mathbb{R}^3)$ it has $$\sum\limits_{j=1}^N\alpha_j\psi(X_j^h(0),V_j^h(0))=(f_0^h,\psi),$$ where $(f_0^h,\psi)$ is the approximation of $\int_{\Omega} f_0\psi dxdv.$
If the weight is considered as a function, it also needs to satisfy the corresponding time equation which has been discussed in~\cite{Chertock2012Convergence}. 
In particular, if we choose
\begin{equation}\label{alpha0top}
\alpha_j=f_0(X_j^h(0),V_j^h(0))\vert\Omega_j\vert
\end{equation}
as the weight of $j$th particle, where $\vert\Omega_j\vert$ is the measure of the interval $\Omega_j$.
Notice that in this case the weight of each particle is constant, i.e., the physical quantity it carries will not change.
On the other hand, we are discussing this from the point of view of a fixed particle generation mesh, and the particle method can also be understood from the point of view of the Monte Carlo approach.
In this case we need to sample the initial macro particles according to the initial distribution $f_0$.
The corresponding error analysis requires the use of statistical tools such as kernel density estimation~\cite{Monaghan1982Why,Qin2018Reducing,Evstatiev2021Noise}.

We first give convergence result for the initialized distribution function.
\begin{proposition}
Assuming that there is a partition $\bigcup_{j=1}^{N}\Omega_j=\Omega$ of the interval $\Omega$ and noting that the maximum diameter of the partitioned intervals is $h=\max\limits_j {diam \Omega_j}$, and that the discrete distributions $f_0^h$ and the weights $\alpha_j$ take the form of the expressions (\ref{f0htop}) and (\ref{alpha0top}), respectively, then the function $f_0^h$ weakly converges to the initial distribution function $f_0$ when $h$ tends to zero.
\end{proposition}

\begin{proof}
For any test functions $\psi\in C_0^{\infty}(\mathbb{R}^{3}\times\mathbb{R}^{3})$,
\begin{align}
\notag&\left\lvert\int_{\Omega}f_0\psi dxdv-\int_{\Omega}f_0^h\psi dxdv\right\rvert\\
\notag=&\left\lvert\sum\limits_{j=1}^{N}\left(\int_{\Omega_j}f_0\psi dxdv-\alpha_j\psi(X_j^h(0),V_j^h(0))\right)\right\rvert\\
\label{f0con}=&\left\lvert\sum\limits_{j=1}^{N}\left(\int_{\Omega_j}f_0\psi dxdv-f_0(X_j^h(0),V_j^h(0))\psi(X_j^h(0),V_j^h(0))\vert\Omega_j\vert\right)\right\rvert.
\end{align}
The term (\ref{f0con}) can be estimated  by the midpoint quadrature error. This implies the term on the left side of the above equality tends to $0$ as $h\to 0$.
\end{proof}

We indicated earlier when giving Dirac sum approximation that these macro particles need to satisfy certain equations, and the following proposition that determines their equations of motion.
\begin{proposition}
The Cauchy problem (\ref{VPapp}) has a week solution in the form  $$f^h(x,v,t)=\sum\limits_{j=1}^N\alpha_j\delta(x-X_j^h(t))\delta(v-V_j^h(t))$$  with $(X_j^h(t), V_j^h(t))$ subjecting the following ODEs
\begin{equation}\label{VP2}
\begin{aligned}
&\frac{dX_j^h}{dt}=V_j^h,\\
&\frac{dV_j^h}{dt}=E^h(X^h_j)+V_j^h\times B(X^h_j).
\end{aligned}
\end{equation}
\end{proposition}

\begin{proof}
For any test functions $\psi\in C_0^{\infty}$, substituting $$f^h=\sum\limits_{j=1}^N\alpha_j\delta(x-X_j^h(t))\delta(v-V_j^h(t))$$ into the weak formulation (\ref{weakform}) leads to
\begin{equation*}
\begin{aligned}
&\sum\limits_{j=1}^{N} \alpha_j\psi(X_j^h(0),V_j^h(0),0)+\sum\limits_{j=1}^{N}\int_0^T \alpha_j\left(\psi_t(X_j^h(t),V_j^h(t),t)+V_j^h(t)\cdot\nabla_x\psi\right.\\
&\left.+(E^h(X^h_j(t),t)+V_j^h\times B(X^h_j(t)))\cdot\nabla_v\psi\right)dt=0.
\end{aligned}
\end{equation*}
Reorganizing the above equation gives
\begin{equation*}
\begin{aligned}
&\sum\limits_{j=1}^{N}\alpha_j\psi(X_j^h(0),V_j^h(0),0)+\sum\limits_{j=1}^{N}\int_0^T\alpha_j\frac{d\psi}{dt}dt-\sum\limits_{j=1}^{N}\int_0^T\alpha_j\left(\frac{dX_j^h(t)}{dt}-V_j^h(t)\right)\cdot\nabla_x\psi dt\\
&-\sum\limits_{j=1}^{N}\int_0^T\alpha_j\left(\frac{dV_j^h(t)}{dt}-E^h(X^h_j(t),t)-V_j^h(t)\times B(X^h_j(t))\right)\cdot\nabla_v\psi dt=0.
\end{aligned}
\end{equation*}
\end{proof}

\subsection{Construction of mollified electric field}

Notice that the Poisson kernel function $K(x,y)$ in equation (\ref{E0}) is not continuous in full space and explodes at $x=y$, i.e. there is a singularity.
Therefore we need to regularize the kernel function to obtain a mollified electrostatic field.
This idea was first introduced in the literature~\cite{Raviart1985An} where it was used to analyze the convergence of the particle method for the $1$+$1$-dimensional Vlasov--Poisson system, and was subsequently extended to the higher dimensional case~\cite{Victory1989On,Victory1991TheS,Victory1991TheF}.

In this section, we describe how to build a mollified electrostatic field.
For our analysis, we set a rectangular grid for the particle method with constant mesh size in a single dimension.
Let $\Delta{x_i},\Delta{v_i},i=1,2,3$ be the mesh widths of the particle grid. For each $j=(j_1,j_2)\in \mathbb{Z}^{6}$, define the grid by
\begin{align*}
A_j=\{&(x,v)\in \mathbb{R}^{6}\lvert(j_{1,i}-1)\Delta{x_i}\leq x_i\leq j_{1,i}\Delta{x_i},\\
&(j_{2,i}-1)\Delta{v_i}\leq v_i\leq j_{2,i}\Delta{v_i},\ i=1,2,3\},
\end{align*} 
and denote $$(x_{j_1},v_{j_2})=\{((j_{1,i}-1/2)\Delta{x_i},(j_{2,i}-1/2)\Delta{v_i}),i=1,2,3\}$$ as the center of $A_j$, and $\beta=\sqrt{\sum\limits_{i=1}^3(\Delta{x_i}^2+\Delta{v_i}^2)}$ is the size of the grid.
Of course, in practical calculation we consider a finite domain, where the indices will be limited.
But in fact the analysis in this chapter still holds for countable case, so for the time being we do not place special emphasis on limiting the computational area.
We approximate the initial function $f_0$ by Dirac sum $$f_0^h=\sum\limits_{j\in\mathbb{Z}^{6}}\alpha_j\delta(x-x_{j_1})\delta(v-v_{j_2})$$ where
$\alpha_j=f_0(x_{j_1},v_{j_2})\mathop{\Pi}\limits_{i=1}^3\Delta{x_i}\Delta{v_i}$.
The solution $f^h$ can be obtained by solving the equation (\ref{VP2}) with the initial condition $$X^h_j(0)=x_{j_1},\quad V^h_j(0)=v_{j_2}.$$

To solve the above equation (\ref{VP2}), as follows we first introduce the regularizing function $\zeta\in L^{1}\cap L^{\infty}$ which has compact support in interval $[-1,+1]^3$ and satisfies $\int_{\mathbb{R}^{3}}\zeta(x)dx=1.$
It is also often referred to as the shape function.
The most common one is
\begin{equation*}
\zeta(x)=
\begin{cases}
(1-\vert x_1\vert)(1-\vert x_2\vert)(1-\vert x_3\vert)&0\le \Vert x\Vert_{\infty}\le 1, \\
0&\mbox{others.}
\end{cases}
\end{equation*}

For a given $r >0$, denote
$\zeta_{r}(x)=\frac{1}{r^3}\zeta(\frac{x}{r})$
and
$K_{r}(x,\cdot)=K(x,\cdot)\ast\zeta_{r}.$
The mollified electric field $E^h$ can be taken as
\begin{equation}\label{E}
\begin{aligned}
E^h(x,t)&=\int_{\mathbb{R}^{3}} K_{r}(x,y)\left(\int_{\mathbb{R}^{3}}f^h(y,v,t)dv-\rho_0\right)dy\\
&=\sum\limits_{j\in\mathbb{Z}^{6}}\alpha_j K_{r}(x,X_j^h(t))-\rho_0 \int_{\mathbb{R}^{3}} K_{r}(x,y)dy.
\end{aligned}
\end{equation}

If $\zeta_{r}$ is even, we have
\begin{align}
\label{E1} E^h(x,t)&=\int_{\mathbb{R}^{3}} K_{r}(x,y)\left(\int_{\mathbb{R}^{3}}f^h(y,v,t)dv-\rho_0\right)dy\\
\label{E2} &=\int_{\mathbb{R}^{3}} K(x,y)\left(\int_{\mathbb{R}^{3}}f_{r}^h(y,v,t)dv-\rho_0\right)dy,
\end{align}
where $f_{r}^h=\sum\limits_{j\in\mathbb{Z}^{6}}\alpha_j\zeta_{r}(x-X_j^h)\delta(v-V_j^h)$.
Substituting $f_{r}^h$ for the equation (\ref{VP02}) and (\ref{VP03}) gives
\begin{align}
\notag&E^h=-\nabla\phi^h,\\
\label{PEM}&-\Delta\phi^h=\sum\limits_{j\in\mathbb{Z}^{6}}\alpha_j\zeta_{r}(x-X_j^h)-\rho_0.
\end{align}
Therefore, we may consider approximation problem for particle method either as a zero-size particle model coupled with a smoothing electrostatic field by equation~(\ref{E1}), or as a finite-size particle model with radius $r$ by equation~(\ref{E2}).

\subsection{Preliminaries}
In this section we lists some known lemmas which we need for the proof as follows. For the proofs, we can refer to the works in \cite{Raviart1985An,Victory1989On}.

\begin{lemma}\label{lemma1}
Assume that the function $\zeta$ satisfies\\
$(1)$ $\int_{\mathbb{R}^{3}}\zeta(x)dx=1$,\\
$(2)$ $\int_{\mathbb{R}^{3}}x^i\zeta(x)dx=0$ for  $1\le\vert i\vert\le k-1$, where $i\in{\mathbb{N}^{3}}$ is the multiple index, and $k$ is an integer,\\
$(3)$ $\int_{\mathbb{R}^{3}}\lvert x\rvert^k\lvert\zeta(x)\rvert dx<+\infty$.\\
Then for some constant $C>0$ and all $g\in W^{k,p}(\mathbb{R}^{3}),1\le p\le+\infty$,
\begin{center}
$\Vert g-g\ast\zeta_{r}\Vert_{L^p(\mathbb{R}^{3})}\le Cr^k \vert g\vert_{k,p,\mathbb{R}^{3}}$.
\end{center}
\end{lemma}

\begin{lemma}\label{lemma2}
We denote $A_j(t)=\{(X(t;x,v,0),V(t;x,v,0))\lvert(x,v)\in A_j\}$.
Assume that $f_0\in W^{m+1,\infty}\cap W^{m,1}$ and $B(x)\in W^{m,\infty}$. Then, for all $T>0$, there exists a constant $C(T)$, such that
\begin{equation*}
card(A_j\cap \mathop{\Pi}\limits_{i=1}^3[a_i,b_i]\times\mathop{\Pi}\limits_{i=1}^3[c_i,d_i])\le C(T)\mathop{\Pi}\limits_{i=1}^3(1+\frac{b_i-a_i}{\beta})(1+\frac{d_i-c_i}{\beta}),
\end{equation*}
where $card(A_j\cap \mathop{\Pi}\limits_{i=1}^3[a_i,b_i]\times\mathop{\Pi}\limits_{i=1}^3[c_i,d_i])$ denotes the number of the intersection of $A_j(t)$ and $\mathop{\Pi}\limits_{i=1}^3[a_i,b_i]\times\mathop{\Pi}\limits_{i=1}^3[c_i,d_i]$.
\end{lemma}

\begin{lemma}\label{lemma3}
Denote $\mathcal{E}_j(g)=\int_{A_j}g(x,v)dxdv-\mathop{\Pi}\limits_{i=1}^3\Delta{x_i}\Delta{v_i}g(x_{j_1},v_{j_2})$. Let $m\ge 1$ be an integer, $p>2n/m$ and $1/p+1/q=1$. Then there is a constant $C>0$ independent of $\Delta{x_i}$, $\Delta{v_i}$ and $j$, such that the following estimation holds.
\begin{equation*}
\lvert\sum_j\mathcal{E}_j(g)\rvert\le C\beta^m(\mathop{\Pi}\limits_{i=1}^3\Delta{x_i}\Delta{v_i})^{1/q}\sum_j\lvert g\rvert_{m,p,A_j}
\end{equation*}
for all $g\in W^{m,p}\cap L^{1}(m\le 2)$ or $g\in W^{m,p}\cap  W^{m,1}(m\ge 3)$.

\end{lemma}

\begin{lemma}\label{lemma4}
Assume that $\zeta\in W^{l,\infty}\cap W^{l,1}$ and $k(x)=\frac{1}{4\pi}\frac{x}{\vert x\vert^3}$. Then, the following inequality holds
\begin{equation*}
\lvert D^{\alpha}k_r\rvert\le Cr^{-(l+2)},
\end{equation*}
for all $\lvert\alpha\rvert=l$ and some constant $C>0$.
\end{lemma}

\begin{lemma}\label{lemma5}
Assume that $\zeta\in L^{1}$ and $k(x)=\frac{1}{4\pi}\frac{x}{\vert x\vert^3}$. Then, for all integers $l\ge 2$ and some constant $C>0$,
\begin{equation*}
\int_{\lvert x\rvert\ge 2r}\lvert D^{\alpha}k_r(x)\rvert dx\le Cr^{-(l-1)}, \lvert\alpha\rvert=l.
\end{equation*}
\end{lemma}

\subsection{Error estimation of particle method}

In this section, we consider the error analysis for the maximal ordering scaling case.  In this case, the external magnetic field is given by $\frac{B(\varepsilon x)}{\varepsilon}$, where $\varepsilon$ is the parameter. 
Denote the error as $$e^h(t)=\max\limits_j(\lvert X_j(t)-X_j^h(t)\rvert+\lvert V_j(t)-V_j^h(t)\rvert)$$ with  $X_j(t)=X(t;x_{j_1},v_{j_2},0)$ and $V_j(t)=V(t;x_{j_1},v_{j_2},0)$ are the corresponding exact solutions.
\begin{theorem}\label{Convergence}
Take the approximation of $E$ as
$$E^h(x,t)=\int_{\mathbb{R}^{3}} K(x,y)\left(\sum\limits_j\alpha_j\zeta_{r}(x-X_j^h)-\rho_0\right)dy,$$
under the assumption in Lemma~\ref{lemma1} and suppose that\\
(a) $f_0,\rho\in W^{s,1}\cap W^{s,\infty}$, $s=\max{(k,7)}$ and $\lvert D^{\alpha}f_0(x,v)\rvert\le C(1+\vert v\vert)^{-\mu}$, $\mu>1$, $\lvert\alpha\rvert\le s$.\\
(b) $\zeta$ is an even  function and belongs to the space $W^{s,1}\cap W^{s,\infty}$.\\
(c) $r=O(\beta^{\gamma}),0<\gamma\le 1$.\\
Then, we have the following estimate
\begin{equation*}
\begin{aligned}
\Vert E-E^h\Vert_{L^\infty}\le C(f_0,\zeta,T)(r^k+\frac{\beta^{s}}{r^{s-1}}+\max\limits_j\vert X_j^h-X_j\vert).
\end{aligned}
\end{equation*}
\end{theorem}

\begin{proof}
By using form (\ref{E0}) and (\ref{E1}), we have
\begin{align}\label{eq:eqE}
E-E^h=&\underbrace{\int_{\mathbb{R}^{3}}\int_{\mathbb{R}^{3}}(K-K_{r})fdvdy}_{\mathcal{A}}+\underbrace{\rho_0 \int_{\mathbb{R}^{3}} K_{r}dy-\rho_0 \int_{\mathbb{R}^{3}} Kdy}_{\mathcal{B}}\\
+&\underbrace{\int_{\mathbb{R}^{3}}\int_{\mathbb{R}^{3}}K_{r}fdvdy-\sum\limits_{j}\alpha_j K_{r}(x,X_j(t))}_{\mathcal{C}}+\underbrace{\sum\limits_{j}\alpha_j(K_{r}(x,X_j(t))-K_{r}(x,X_j^h(t)))}_{\mathcal{D}}.
\end{align}

Due to the symmetry of $\zeta_{r}$, and $\rho$ has the form in (\ref{VP04}) we have
\begin{align}\label{eq:eqA}
\mathcal{A}=&k\ast(\rho-\rho\ast\zeta_{r}),
\end{align}
where $k(x)=\frac{1}{4\pi}\frac{x}{\vert x\vert^3}$.
Denote $B(0,1)$ the ball with center $0$ and radius $1$ and $B(0,1)^c$ is its complementary set, $k$ can be expressed as
$k={\bf{1}}_{B(0,1)}k+{\bf{1}}_{B(0,1)^c}k.$
It is clear that ${\bf{1}}_{B(0,1)}k=O(\vert x\vert^{-2}{\bf{1}}_{\vert x\vert\le 1})\in L^1(\mathbb{R}^{3})$ and  ${\bf{1}}_{B(0,1)^{c}}k=O(\vert x\vert^{-2}{\bf{1}}_{\vert x\vert\ge 1})\in L^{\infty}(\mathbb{R}^{3})$. Moreover, $\rho\in L^1(\mathbb{R}^{3})$ due to the global mass conservation.
Thus, by applying the Young's inequality  to (\ref{eq:eqA}) and  using Lemma~\ref{lemma1},  one concludes that
\begin{align*}
\Vert\mathcal{A}\Vert_{L^{\infty}}=&\Vert k\ast(\rho-\rho\ast\zeta_{r})\Vert_{L^{\infty}}\\
\le&C(\Vert\rho-\rho\ast\zeta_{r}\Vert_{L^{1}}+\Vert\rho-\rho\ast\zeta_{r}\Vert_{L^{\infty}})\\
\le&C(f_0)r^k.
\end{align*}

For term $\mathcal{B}$ in (\ref{eq:eqE}), with the use of Lemma~\ref{lemma1} it can  derived directly that $\mathcal{B}=0$.

According to (\ref{sol1}) and (\ref{sol2}), we have
\begin{equation*}
\begin{aligned}
&\int_{\mathbb{R}^{3}}\int_{\mathbb{R}^{3}}K_{r}(x,y)f(y,v,t)dvdy\\
=&\int_{\mathbb{R}^{3}}\int_{\mathbb{R}^{3}}K_{r}(x,y)f_0(X(0;y,v,t),V(0;y,v,t))dvdy\\
=&\int_{\mathbb{R}^{3}}\int_{\mathbb{R}^{3}}K_{r}(x,X(t;y,v,0))f_0(y,v)dvdy.
\end{aligned}
\end{equation*}
The last equation holds due to $X(0;X(t;y,v,0),V(t;y,v,0),t)=y$.
Thus, term $\mathcal{C}$ can be expressed as
\begin{align*}
\mathcal{C}=&\int_{\mathbb{R}^{3}}\int_{\mathbb{R}^{3}}K_{r}(x,X(t;y,v,0))f_0(y,v)dvdy-\sum\limits_{j}\alpha_j K_{r}(x,X(t;x_{j_1},v_{j_2},0))\\
=&\sum\limits_{j}\mathcal{E}_j(f_0K_r(x,X(t;\cdot,\cdot,0))).
\end{align*}
We split $\{j\,\lvert j\in \mathbb{Z}^6\}$ as two sets, it follows from the above equality that
\begin{align}\label{eq:eqC}
\mathcal{C}=\sum\limits_{j\in I_1\times\mathbb{Z}^3}\mathcal{E}_j(f_0K_r(x,X(t;\cdot,\cdot,0)))+\sum\limits_{j\in I_2\times\mathbb{Z}^3}\mathcal{E}_j(f_0K_r(x,X(t;\cdot,\cdot,0))),
\end{align}
where $I_1=\{j_1\,\lvert A_{j_1}(t)\cap\mathop{\Pi}\limits_{i=1}^3\otimes[x_i-r,x_i+r]\neq\varnothing \}$ and $I_2=\mathbb{Z}^3\backslash I_1$.
Then we introduce $$J(t,\xi)=\{(j_1,j_2)\lvert j_1\in I_1, \xi_i\le v_{j_2,i}\le\xi_i+1,  i=1,2,3\}.$$ Thus, if $j\in J(t,\xi)$, $A_j(t)$ intersects the cuboid $$\mathop{\Pi}\limits_{i=1}^3\otimes[x_i-r,x_i+r]\times\mathop{\Pi}\limits_{i=1}^3\otimes[\xi_i-C(T),\xi_i+1+C(T)]$$ with $C(T)$ a constant depending on $T$.
Hence, by Lemma~\ref{lemma2}, we have $$card(J(t,\xi))\le C(T)(r+\beta)^3\beta^{-6}.$$
To estimate the first term of (\ref{eq:eqC}), we use Lemmas \ref{lemma3} and \ref{lemma4}. Under the assumption of $f_0$, one concludes
\begin{equation}\label{C1}
\begin{aligned}
&\left\lvert\sum\limits_{j\in I_1\times\mathbb{Z}^3}\mathcal{E}_j(f_0K_r(x,X(t;\cdot,\cdot,0)))\right\rvert\\
\le&C\beta^{s}\mathop{\Pi}\limits_{i=1}^3\Delta{x_i}\Delta{v_i}\sum_{\xi\in\mathbb{Z}^3}\sum_{j\in J(t,\xi)}\vert f_0K_r(x,X(t;x_{j_1},v_{j_2},0))\vert_{s,\infty}\\
\le&C(f_0,\zeta,T)\sum_{\lvert l\rvert=0}^s(\beta^{s}\mathop{\Pi}\limits_{i=1}^3\Delta{x_i}\Delta{v_i}(r+\beta)^3\beta^{-6}\sum_{\xi\in\mathbb{Z}^3}(1+\lvert\xi\rvert)^{-\gamma}\lvert D^{l}k_{r}\rvert)\\
\le&C(f_0,\zeta,T)\beta^{s}(1+\sum_{\lvert l\rvert=1}^s\frac{1}{r^{l-1}}(1+\frac{\beta}{r})^3)\\
\le&C(f_0,\zeta,T)(\frac{\beta}{r})^{s-1}(1+\frac{\beta}{r})^3\beta
\end{aligned}
\end{equation}
Furthermore, employing Lemmas \ref{lemma3} and \ref{lemma5}, we have
\begin{equation}\label{C2}
\begin{aligned}
&\left\lvert\sum\limits_{j\in I_2\times\mathbb{Z}^3}\mathcal{E}_j(f_0K_r(x,X(t;\cdot,\cdot,0)))\right\rvert\\
\le&C\beta^s\lvert f_0K_r(x,X(t;\cdot,\cdot,0))\rvert_{s,1}\\
\le&C(f_0,\zeta)\frac{\beta^s}{r^{s-1}}.
\end{aligned}
\end{equation}
Combing (\ref{C1}) and (\ref{C2}) gives
\begin{equation*}
\begin{aligned}
\lvert \mathcal{C}\rvert\le C(f_0,\zeta)\frac{\beta^s}{r^{s-1}}.
\end{aligned}
\end{equation*}

For term $\mathcal{D}$, due to $\zeta_r\in L^1\cap L^{\infty}$, it has
\begin{align*}
\Vert\mathcal{D}\Vert_{L^{\infty}}&=\sum\limits_{j}\alpha_j\Vert K_{r}(x,X_j(t))-K_{r}(x,X_j^h(t))\Vert_{L^{\infty}}\\
&\le\sum\limits_{j}\alpha_j(\Vert\zeta_{r}(x-X_j(t))-\zeta_{r}(x-X_j^h(t))\Vert_{L^1}+\Vert\zeta_{r}(x-X_j(t))-\zeta_{r}(x-X_j^h(t))\Vert_{L^{\infty}}).
\end{align*}
Consider the support of $\zeta_{r}$, we should have $\vert x-X_j\vert\le r+\max\limits_j\vert X_j^h-X_j\vert$  which implies that $X_j\in B(x,r+\max\limits_j\vert X_j^h-X_j\vert)$.
We denote the set $J'(t,\xi)$ for $$X_j\in B(x,r+\max\limits_j\vert X_j^h-X_j\vert)$$ and $\xi_i\le v_{j_2,i}\le\xi_i+1,i=1,2,3$.
Applied Lemma~\ref{lemma2} again, we have $$card(J'(t,\xi))\le C(T)(r+\beta+\max_j\vert X_j^h-X_j\vert)^{3}\beta^{-6}.$$
Notice that $\Vert\zeta_r\Vert_{L^1}\le C,\Vert\zeta_r\Vert_{L^{\infty}}\le Cr^{-3}$, then we have
\begin{equation*}
\begin{aligned}
&\sum\limits_{\xi\in\mathbb{Z}^3}\sum\limits_{j\in J'(t,\xi)}\alpha_j(\Vert\zeta_{r}(x-X_j(t))-\zeta_{r}(x-X_j^h(t))\Vert_{L^1}+\Vert\zeta_{r}(x-X_j(t))-\zeta_{r}(x-X_j^h(t))\Vert_{L^{\infty}})\\
\le& C(T,f_0,\zeta)\mathop{\Pi}\limits_{i=1}^3\Delta{x_i}\Delta{v_i}(r+\beta+\max_j \lvert X_j^h-X_j\rvert)^{3}\beta^{-6}\sum_{\xi\in\mathbb{Z}^3}(1+\lvert\xi\rvert)^{-\gamma}\left(1+\frac{\max\limits_j\lvert X_j^h-X_j\rvert}{r^3}\right)\\
\le& C(T,f_0,\zeta)(1+\frac{\beta}{r})^3\max_j \lvert X_j^h-X_j\rvert.
\end{aligned}
\end{equation*}

In summary, we obtain the following estimate
 $$\Vert E-E^h\Vert_{L^\infty}\le C(f_0,\zeta,T)(r^k+\frac{\beta^{s}}{r^{s-1}}+\max\limits_j\vert X_j^h-X_j\vert),$$
where we  use the fact $r=O(\beta^{\gamma}),0<\gamma\le 1$.

\end{proof}


\noindent{\bf Remark.} Plasma is a quasi-neutral matter and has Derby shielding effect.
Thus, in practical calculation a sufficiently large number of macro particles are needed to satisfy this property.
This is also the source of assumption (c).
The following figures illustrate the importance of the assumption.
\begin{figure}[h!]
\centering
\subfigure[]{
\includegraphics[scale=.35]{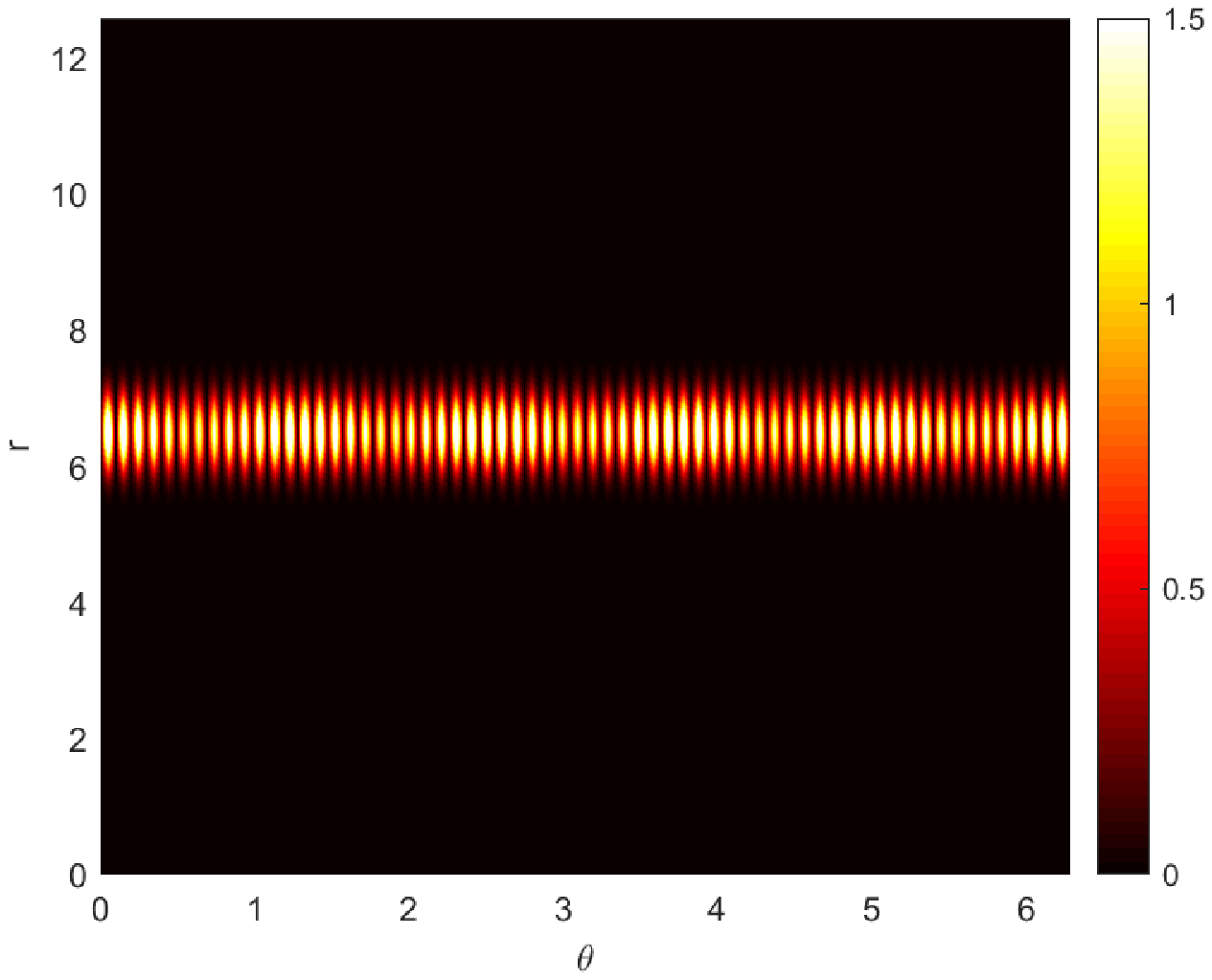}
}
\quad
\subfigure[]{
\includegraphics[scale=.35]{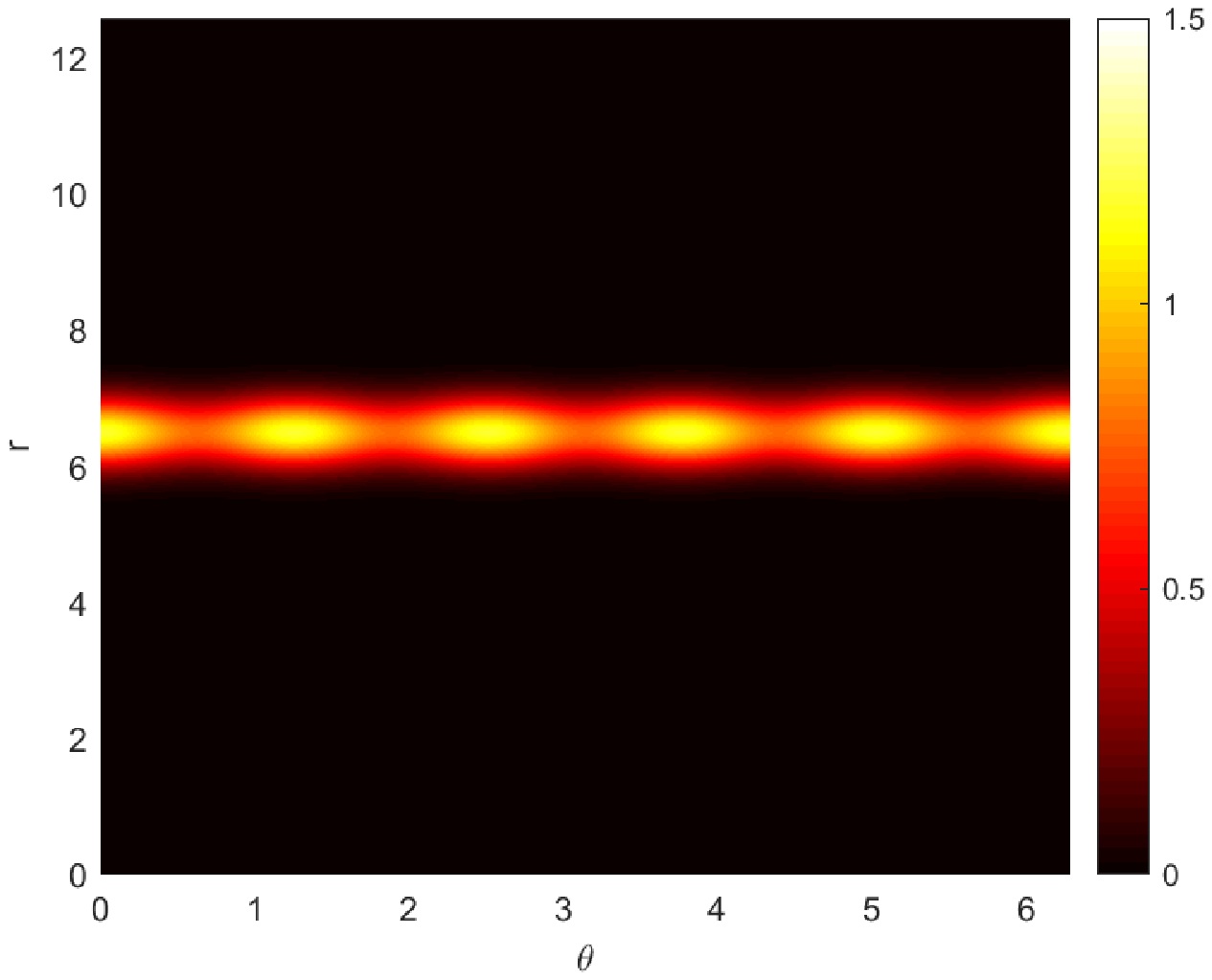}
}
\caption{Influence of macro particles' number on numerical simulation. (a):$r\ll\beta$; (b):$r=O(\beta)$.}
\end{figure}

\begin{theorem}\label{Convergence2}
Under assumption of Theorem~\ref{Convergence}, if $f_0$ has a compact support about $v$ and $E,B\in W^{1,\infty}(\mathbb{R}^{3})$, we have
\begin{equation*}
\begin{aligned}
e^h(t)\le C(f_0,\zeta,E,B,T)(r^k+\frac{\beta^{s}}{r^{s-1}}).
\end{aligned}
\end{equation*}
\end{theorem}

\begin{proof}
Fix $j$ and denote $z=X_j-X_j^h$, $w=V_j-V_j^h$, we have
\begin{align}
\label{eq:eqz}\dot{z}&=w\\
\notag\dot{w}&=E(X_j)-E^h(X_j^h)+V_j\times\frac{B(\varepsilon X_j)}{\varepsilon}-V_j^h\times\frac{B(\varepsilon X_j^h)}{\varepsilon}\\
\label{eq:eqw}&=w\times\frac{B(0)}{\varepsilon}+g
\end{align}
where 
\begin{align*}
g&=E(X_j)-E^h(X_j)+E^h(X_j)-E^h(X_j^h)\\
&+\frac{1}{\varepsilon}V_j^h\times(B(\varepsilon X_j)-B(\varepsilon X_j^h))+\frac{1}{\varepsilon}w\times (B(\varepsilon X_j)-B(0)).
\end{align*}
Integrating (\ref{eq:eqz}) w.r.t $t$ gives
\begin{equation}\label{eq:eqz1}
z(t)=\int_0^t w(s)ds.
\end{equation}
Applying the constant-variation method to (\ref{eq:eqw}) leads to
\begin{equation}\label{eq:eqw1}
w(t)=\int_0^t e^{(t-s)\hat{B}(0)/\varepsilon}g(s)ds.
\end{equation}
Due to that $E$ and $B$ is Lipschitz continuous, we have
\begin{equation*}
\begin{aligned}
&\lvert E^h(X_j)-E^h(X_j^h)\rvert\le C(E)\lvert z(s)\rvert,\\
&\left\vert\frac{1}{\varepsilon} V_j^h(s)\times(B(\varepsilon X_j(s))-B(\varepsilon X_j^h(s)))\right\vert\le C(B)\vert z(s)\vert,\\
&\left\vert\frac{1}{\varepsilon} w(s)\times (B(\varepsilon X_j(s))-B(0))\right\vert\le C(B)\vert w(s)\vert.
\end{aligned}
\end{equation*}

It follows from (\ref{eq:eqz1}) and (\ref{eq:eqw1}) that
\begin{equation*}
\begin{aligned}
\vert z(t)\vert+\vert w(t)\vert\le\int_0^t\Vert E(X_j(s))-E^h(X_j(s))\Vert_{L^{\infty}}ds+C(E,B)\int_0^t(\vert z(s)\vert+\vert w(s)\vert)ds.
\end{aligned}
\end{equation*}

By Gronwall's inequality, we obtain
\begin{equation*}
\begin{aligned}
e^h(t)&\le\int_0^t\Vert E(X_j(s))-E^h(X_j(s))\Vert_{L^{\infty}}ds+C(E,B)\int_0^t e^h(s)ds\\
&\le C(f_0,\zeta,E,B,T)(r^k+\frac{\beta^{s}}{r^{s-1}}).
\end{aligned}
\end{equation*}
\end{proof}

\noindent{\bf Remark.} If $f_0$ is not compact supported, we will choose the appropriate computational domain in practical calculations.
Supposing that we choose $\vert v_{j_2}\vert\le M$ for a given $M$, then $$\Vert E-E^h\Vert_{L^\infty}\le C(f_0,\zeta,T)\left(r^k+\frac{\beta^{s}}{r^{s-1}}+\max\limits_j\vert X_j^h-X_j\vert+\frac{1}{M^{\mu-1}}\right).$$
This can be derived from the assumption (a).

With above discussion, we have the following result.
\begin{theorem}\label{TH1}
Assume the conditions in Theorem ~\ref{Convergence2} are satisfied, then for all $\psi\in C_0^{\infty}$, $r\to 0$,
\begin{equation*}
\begin{aligned}
\int_{\mathbb{R}^{3}}\int_{\mathbb{R}^{3}}(f-f^h)\psi dxdv\to 0.
\end{aligned}
\end{equation*}
\end{theorem}
That is, the discrete distribution $f^h$ converges weakly to the distribution function $f$.

\section{Hamiltonian method of charged-particle equation}
In this section, we discuss the convergence results for fully discretized system by Hamiltonian particle methods under the maximal ordering scaling.
To simplify the notation, we ignore the subscript or superscript, and consider the following charged-particle dynamics under a strong magnetic field.
\begin{equation}
\begin{aligned}
&\dot{x}=v,\\
&\dot{v}=v\times\frac{B(\varepsilon x)}{\varepsilon}+E^h(x),\quad 0<t\le T,\\
&x(0)=x_0,\quad v(0)=v_0.
\end{aligned}
\label{CPD}
\end{equation}

\subsection{Recall of Hamiltonian method}

For spatial discretization, we use finite element methods to solve the Poisson equation (\ref{PEM}) with the Dirichlet boundary condition.
The approximate problem is to find $\phi_{h}\in V_{h}\subset H_{0}^{1}(\Omega_{x})$
such that
\begin{equation*}
\left(\nabla\phi_{h},\mathbf{\nabla}\psi_{h}\right)=\left<\rho^h,\psi_{h}\right>,\quad\forall \psi_{h}\in V_{h},
\end{equation*}
where $\rho^h=\sum\limits_j\alpha_j\zeta_{r}(x-X_j^h)-\rho_0$.
The inner product is defined by $\left(f,g\right)=\int_{\Omega_{x}}f\cdot g dx$ on space $L^{2}(\Omega_{x})$ and $V_{h}$ is the finite-dimensional subspace. In particular, for all $\psi_{h}\in H_{0}^{1}(\Omega_{x})$ and $\rho^h\in L^{2}(\Omega_{x})\subset H^{-1}(\Omega_{x})$, we have $\left(\rho^h,\psi_{h}\right)=\left<\rho^h,\psi_{h}\right>$,
and the solution to the variational problem is unique.
Consider the spatial discretization, we need to express
\begin{equation*}
E^h-\tilde{E}=E^h+\nabla\Pi\phi^h.
\end{equation*}
Here, $\Pi$ is the projection to the finite element space and $\tilde{E}:=-\nabla\Pi\phi^h$. We can use the following estimates of conforming finite element.
\begin{lemma}\label{FEM}
Suppose that the grid is a regular mesh\footnote{The more details about the regular mesh can be found in~\cite{Brenner2008}.} and $\phi^h\in H^{k+1}(\Omega_x)\cap H_0^1(\Omega_x)$, for $k$-th order finite element method, we have
\begin{equation*}
\begin{aligned}
\Vert\phi^h-\Pi\phi^h\Vert_{1,\Omega_x}\le Ch_x^k\vert\phi^h\vert_{k+1,\Omega_x},
\end{aligned}
\end{equation*}
where $h_x$ is the size of gird.
\end{lemma}
When $B$ is constant, we split the Hamiltonian as
\begin{align*}
 & H=H_v+H_e, \quad
  \text{where~} H_v=\frac{1}{2} \sum_{j} \alpha_j (V_j^h)^2, \quad H_e=\frac{1}{2}\int \tilde{E}\cdot\tilde{E} dx.
\end{align*}
With each part of the Hamiltonian, we can split the system into two parts, and each subsystem possesses the Poisson bracket structure. Thus the composition of solutions to each subsystem preserves the Poisson bracket of the original system due the group property of Poisson bracket structure.
If $B$ is nonhomogeneous, we need to split the kinetic component further due to Cartesian coordinate.

For two arbitrary functions $F$ and $G$ of $X_j^h,V_j^h$, the discrete Poisson bracket from the continuous bracket (\ref{eq:MMWB}) reads
\begin{equation*}
\begin{aligned}
\left\{ F,G\right\}
&=\sum_{j}\frac{1}{\alpha_{j}}\left(\frac{\partial F}{\partial X_j^h}\cdot\frac{\partial G}{\partial V_j^h}-\frac{\partial G}{\partial X_j^h}\cdot\frac{\partial F}{\partial V_j^h}\right)\\
&+\sum_{j} \frac{1}{\varepsilon\alpha_{j}}{B}(\varepsilon X_j^h)\cdot\left(\frac{\partial F}{\partial V_j^h} \times\frac{\partial G}{\partial V_j^h}\right).
\end{aligned}
\end{equation*}
Associated with the Hamiltonian
$H_{v}$, the subsystem is $\dot{F}=\left\{ F,H_{v}\right\} $. For the $j$-th particle it is,
\begin{equation}
\begin{aligned}
& \dot{X}_{j}^h=V_{j}^h,\\
& \dot{V}_{j}^h=\frac{1}{\varepsilon}V_{j}^h\times B(\varepsilon X_{j}^h).
\end{aligned}
\label{eq:dSysHE}
\end{equation}
The equation $\dot{F}=\left\{ F,H_{e}\right\} $ associated with the Hamiltonian $H_{e}$ is
\begin{equation}
\begin{aligned}
& \dot{X}_{j}^h=0,\\
& \dot{V}_{j}^h=\tilde{E}(X_{j}^h).
\end{aligned}
\label{eq:dSysHB}
\end{equation}
With the given exact solutions to the subsystems,
Poisson integrators can be derived by the compositions of the sub-flows~(\ref{eq:dSysHE}) and (\ref{eq:dSysHB}).

Following the idea similar to it in~\cite{ZXF2021Error}, we establish the convergence result for our numerical scheme.
We introduce the time rescaling $$\tau:=\frac{t}{\varepsilon},\quad z(\tau):=x(t),\quad w(\tau):=v(t),\quad 0\le \tau\le \frac{T}{\varepsilon}.$$
Then equation (\ref{CPD}) turns into a long-time problem
\begin{equation}
\begin{aligned}
&\dot{z}(\tau)=\varepsilon w(\tau),\\
&\dot{w}(\tau)=w(\tau)\times B(\varepsilon z(\tau))+\varepsilon E^h(z(\tau)),\quad 0<\tau\le \frac{T}{\varepsilon},\\
&z(0)=x_0,\quad w(0)=v_0.
\end{aligned}
\label{CPD2}
\end{equation}
We denote the step size $\tilde{h}=\Delta{\tau}$, time grids $\tau_n=n\tilde{h}$, and the numerical solutions $$z^n\approx z(\tau_n),\quad w^n\approx w(\tau_n),\quad z^0=x_0,\quad w^0=v_0.$$
Moreover, we use $T_0$ to denote the single period of $e^{\tau\hat{B}(0)}$
where
\begin{equation}\label{hatB}
\hat{B}=\left(
            \begin{array}{ccc}
              0 & B_z & -B_y \\
              -B_z & 0 & B_x \\
              B_y & -B_x & 0 \\
            \end{array}
          \right)
\end{equation}
for $B=(B_x,B_y,B_z)$.

Noting that $\vert B(x)\vert=b(x)$, we first consider the following fully discrete scheme
\begin{equation}
\begin{aligned}
z^{n+1}&=z^n+\varepsilon\int_0^{\tilde{h}} e^{s\hat{B}(\varepsilon z^n)}ds\,w^n\\
&=z^n+\varepsilon\left(\tilde{h}I+\frac{1-\cos(\tilde{h}b)}{b^2}\hat{B}(\varepsilon z^n)+\frac{\tilde{h}b-\sin{\tilde{h}b}}{b^3}\hat{B}^2(\varepsilon z^n)\right)w^n,\\
w^{n+1}&=e^{\tilde{h}\hat{B}(\varepsilon z^n)}w^n+\varepsilon \tilde{h}\tilde{E}(z^{n+1})\\
&=\left(I+\frac{\sin{\tilde{h}b}}{b}\hat{B}(\varepsilon z^n)+\frac{1-\cos{\tilde{h}b}}{b^2}\hat{B}^2(\varepsilon z^n))\right)w^n+\varepsilon\tilde{h}\tilde{E}(z^{n+1}).
\end{aligned}
\label{SCPD}
\end{equation}

Consider the truncated system of (\ref{CPD2}) as
\begin{equation}
\begin{aligned}
&\dot{\tilde{z}}^n(s)=\varepsilon \tilde{w}^n(s),\quad 0<s\le\tilde{h},\\
&\dot{\tilde{w}}^n(s)=\tilde{w}^n(s)\times B(\varepsilon z(\tau_n))+\varepsilon E^h(\tilde{z}^n(s)),\\
&\tilde{z}^n(0)=z(\tau_n),\quad \tilde{w}^n(0)=w(\tau_n).
\end{aligned}
\label{CPDl}
\end{equation}
By denoting the differences $$\eta_z^n(s):=z(\tau_n+s)-\tilde{z}^n(s),\quad\eta_w^n(s):=w(\tau_n+s)-\tilde{w}^n(s),\quad 0\le n<\frac{T}{\varepsilon\tilde{h}},$$
equations (\ref{CPD2})-(\ref{CPDl}) leads to
\begin{equation}
\begin{aligned}
&\dot{\eta}_z^n(s)=\varepsilon \eta_w^n(s),\\
&\dot{\eta}_w^n(s)=\eta_w^n(s)\times B(\varepsilon z(\tau_n))+\varepsilon E^h(z(\tau_n+s))-\varepsilon E^h(\tilde{z}^n(s))+\xi_0^n(s),\\
&\eta_z^n(0)=\eta_w^n(0)=0,
\end{aligned}
\label{CPDel}
\end{equation}
where $\xi_0^n(s)=w(\tau_n+s)\times(B(\varepsilon z(\tau_n+s))-B(\varepsilon z(\tau_n)))$.
Through Taylor expansion and the variation-of-constants formula,we can get the following result.
\begin{lemma}\label{eta}
\cite{ZXF2021Error}\, If $E^h\in C^1(\mathbb{R}^{3})$ and $B\in W^{1,\infty}(\mathbb{R}^{3})$, we can have
\begin{equation}\label{tre}
\vert\eta_z^n(\tilde{h})\vert\le C\varepsilon^3{\tilde{h}}^3,\quad \vert\eta_w^n(\tilde{h})\vert\le C\varepsilon^2{\tilde{h}}^2.
\end{equation}
\end{lemma}
And it is easy to check that the magnetic field $\frac{B(\varepsilon x)}{\varepsilon}$ has the following properties.
\begin{lemma}\label{mang}
The matrix (\ref{hatB}) w.r.t. $B$ satisfies
\begin{equation*}
\hat{B}^3=-b^2\hat{B},\quad (B\cdot y)B=\hat{B}^2 y+b^2 y
\end{equation*}
for any vector $y\in\mathbb{R}^{3}$.

If $B\in W^{1,\infty}(\mathbb{R}^{3})$, we have
\begin{equation*}
\vert B(\varepsilon z)-B(0)\vert\le C\varepsilon,\quad \vert e^{s\hat{B}(\varepsilon z)}-e^{s\hat{B}(0)}\vert\le Cs\varepsilon.
\end{equation*}
\end{lemma}

\subsection{Main results}
We firstly prove a coarse estimate. By using the Lemma~\ref{eta}, the principal question is to evaluate differences between the exact solutions of equations (\ref{CPDl}) and the numerical solutions. It is worth mentioning that the exact solutions of (\ref{CPD}) belong to the space $L^{\infty}(0,T)$ under the assumption of the smoothness about electromagnetic field. We omit the proof of the boundedness about the numerical solutions and it can be checked by induction during the following proof.
\begin{theorem}\label{lemmaCPD}
Assume that $E^h\in C^{1}(\mathbb{R}^{3})$ and $B\in W^{1,\infty}(\mathbb{R}^{3})$, then we have
\begin{equation*}
\begin{aligned}
&\vert z^n-z(\tau_n)\vert+\vert w^n-w(\tau_n)\vert\le C(\tilde{h}+h_x^k),\quad 0\le n\le\frac{T}{\varepsilon\tilde{h}},
\end{aligned}
\end{equation*}
for $\tilde{h}\le \tilde{h}_0$ where $\tilde{h}_0$ is a constant independent of $\varepsilon$.
\end{theorem}

\begin{proof}
By denoting the numerical error $$e_z^{n+1}:=z(\tau_{n+1})-z^{n+1},\quad e_w^{n+1}:=w(\tau_{n+1})-w^{n+1},\quad 0\le n<\frac{T}{\varepsilon\tilde{h}},$$
they can been expressed as $$e_z^{n+1}=\tilde{e}_z^n+\eta_z^n(\tilde{h}),\quad e_w^{n+1}=\tilde{e}_w^n+\eta_w^n(\tilde{h})$$ with the help of the truncated error (\ref{tre}) and $$\tilde{e}_z^n=\tilde{z}^n(\tilde{h})-z^{n+1},\quad \tilde{e}_w^n=\tilde{w}^n(\tilde{h})-w^{n+1},\quad 0\le n<\frac{T}{\varepsilon\tilde{h}}.$$
Then we turn to estimate $\tilde{e}_z^n$ and $\tilde{e}_w^n$.
First of all we define the local truncation error $\xi_z^n$ and $\xi_w^n$ from (\ref{SCPD}) which reads
\begin{equation}\label{ts}
\begin{aligned}
&\tilde{z}^n(\tilde{h})=z(\tau_n)+\varepsilon \int_0^{\tilde{h}} e^{s\hat{B}(\varepsilon z(\tau_n))}ds\,w(\tau_n)+\xi_z^n\\
&\tilde{w}^n(\tilde{h})=e^{\tilde{h}\hat{B}(\varepsilon z(\tau_n))}w(\tau_n)+\varepsilon \tilde{h}\tilde{E}(\tilde{z}^n(\tilde{h}))+\xi_w^n.
\end{aligned}
\end{equation}
By the variation-of-constants formula of the truncated system (\ref{CPDl}), it has
\begin{equation}\label{voc}
\begin{aligned}
&\tilde{z}^n(\tilde{h})=z(\tau_n)+\varepsilon\int_0^{\tilde{h}}\tilde{w}^n(s)ds\\
&\tilde{w}^n(\tilde{h})=e^{\tilde{h}\hat{B}(\varepsilon z(\tau_n))}w(\tau_n)+\varepsilon \int_0^{\tilde{h}}e^{(\tilde{h}-s)\hat{B}(\varepsilon z(\tau_n))}E^h(\tilde{z}^n(s)))ds.
\end{aligned}
\end{equation}
Then (\ref{voc}) implies $$\tilde{z}^n(\tilde{h})=z(\tau_n)+\varepsilon \int_0^{\tilde{h}} e^{s\hat{B}(\varepsilon z(\tau_n))}ds\,w(\tau_n)+\varepsilon^2 \int_0^{\tilde{h}} \int_0^s e^{(s-\sigma)\hat{B}(\varepsilon z(\tau_n))}E^h(\tilde{z}^n(\sigma))d\sigma ds.$$
By Taylor expansion, we can further obtain
\begin{align*}
&\varepsilon\int_0^{\tilde{h}} e^{(\tilde{h}-s)\hat{B}(\varepsilon z(\tau_n))}E^h(\tilde{z}^n(s))ds\\
=&\varepsilon\int_0^{\tilde{h}} (I-(s-\tilde{h})e^{(\tilde{h}-\tau_s^h)\hat{B}(\varepsilon z(\tau_n))}\hat{B}(\varepsilon z(\tau_n)))E^h(\tilde{z}^n(s))ds,
\end{align*}
where $\tau_s^h\in [s,\tilde{h}]$.
Notice that $\xi_w^n=\varepsilon\int_0^{\tilde{h}} e^{(\tilde{h}-s)\hat{B}(\varepsilon z(\tau_n))}E^h(\tilde{z}^n(s))ds-\varepsilon\tilde{h}\tilde{E}(\tilde{z}^n(\tilde{h}))$, we can have
\begin{align*}
\xi_w^n&=\varepsilon\int_0^{\tilde{h}} E^h(\tilde{z}^n(s))ds-\varepsilon\tilde{h}E^h(\tilde{z}^n(\tilde{h}))\\
&-\varepsilon\int_0^{\tilde{h}}(s-\tilde{h})e^{(\tilde{h}-\tau_s^h)\hat{B}(\varepsilon z(\tau_n))}\hat{B}(\varepsilon z(\tau_n))E^h(\tilde{z}^n(s))ds\\
&+\varepsilon\tilde{h}E^h(\tilde{z}^n(\tilde{h}))-\varepsilon\tilde{h}\tilde{E}(\tilde{z}^n(\tilde{h})).
\end{align*}
Together with Lemma~\ref{FEM}, we have $$\vert\xi_w^n\vert\le C\varepsilon{\tilde{h}}(\tilde{h}+h_x^k),\quad 0\le n<\frac{T}{\varepsilon{\tilde{h}}}.$$
Moreover, we can directly obtain that
\begin{equation*}
\xi_z^n=\varepsilon^2 \int_0^{\tilde{h}}\int_0^s e^{(s-\sigma)\hat{B}(\varepsilon z(\tau_n))}\tilde{E}(\tilde{z}^n(\sigma))d\sigma ds.
\end{equation*}
Then we have $$\vert\xi_z^n\vert\le C\varepsilon^2{\tilde{h}}^2.$$

Furthermore, (\ref{ts})-(\ref{SCPD}) leads to
\begin{equation}\label{erroreq}
\begin{aligned}
&e_z^{n+1}=e_z^{n}+\varepsilon \int_0^{\tilde{h}} e^{s\hat{B}(\varepsilon z(\tau_n))}ds\,e_w^{n}+\xi_z^n+\eta_z^n+\zeta_z^n,\\
&e_w^{n+1}=e^{\tilde{h}\hat{B}(\varepsilon z(\tau_n))}e_w^{n}+\xi_w^n+\eta_w^n+\zeta_w^n
\end{aligned}
\end{equation}
where
\begin{align*}
\zeta_z^n=&(\varepsilon \int_0^{\tilde{h}} e^{s\hat{B}(\varepsilon z(\tau_n))}ds-\varepsilon \int_0^{\tilde{h}} e^{s\hat{B}(\varepsilon z^n)}ds)w^{n},\\
\zeta_w^n=&(e^{\tilde{h}\hat{B}(\varepsilon z(\tau_n))}-e^{\tilde{h}\hat{B}(\varepsilon z^n)})w^n+\varepsilon\tilde{h}\tilde{E}(\tilde{z}^n(s))-\varepsilon\tilde{h}\tilde{E}(z^{n+1}).
\end{align*}
It is direct to observe that
\begin{align*}
\vert\zeta_z^n\vert\le&C\varepsilon^2{\tilde{h}}^2\vert e_z^n\vert,\\
\vert\zeta_w^n\vert\le&C\varepsilon\tilde{h}(\vert e_z^n\vert+\vert e_z^{n+1}\vert+\vert\eta_z^n\vert).
\end{align*}
Then (\ref{erroreq}) implies
\begin{equation}\label{errneed}
\begin{aligned}
\vert e_z^{n+1}\vert\le&\vert e_z^{n}\vert+\varepsilon{\tilde{h}}\vert e_w^n\vert+\vert\xi_z^n\vert+\vert \eta_z^n\vert+\vert\zeta_z^n\vert\\
\vert e_w^{n+1}\vert\le&\vert e_w^n\vert+\vert\xi_w^n\vert+\vert\eta_w^n\vert+\vert\zeta_w^n\vert.
\end{aligned}
\end{equation}
Furthermore, we can get
\begin{align*}
&\vert e_z^{n+1}\vert+\vert e_w^{n+1}\vert-\vert e_z^{n}\vert-\vert e_w^n\vert\\
\le&C\varepsilon{\tilde{h}}(\vert e_w^n\vert+\vert e_z^{n}\vert+\vert e_z^{n+1}\vert)+\vert\xi_z^n\vert+\vert\xi_w^n\vert+\vert\eta_z^n\vert+\vert\eta_w^n\vert,
\end{align*}
for $0\le n\le m$.
By noting that $e_z^{0}=e_w^{0}=0$, we have
$$\vert e_z^{m+1}\vert+\vert e_w^{m+1}\vert\le C\varepsilon{\tilde{h}}\sum\limits_{n=0}^m(\vert e_w^n\vert+\vert e_z^n\vert+\vert e_z^{n+1}\vert)+\sum\limits_{n=0}^m(\vert\xi_z^n\vert+\vert\xi_w^n\vert+\vert\eta_z^n\vert+\vert\eta_w^n\vert).$$
Due to $m\varepsilon\tilde{h}\le T$ and Lemma~\ref{eta}, we can obtain $$\vert e_z^{m+1}\vert+\vert e_w^{m+1}\vert\le C\varepsilon\tilde{h}(\sum\limits_{n=0}^m(\vert e_w^n\vert+\vert e_z^{n}\vert+\vert e_z^{n+1}\vert)+\tilde{h}+h_x^k).$$
By Gronwall's inequality, we finally get $$\vert e_z^{m+1}\vert+\vert e_w^{m+1}\vert\le C(\tilde{h}+h_x^k),0\le m<\frac{T}{\varepsilon\tilde{h}}.$$
\end{proof}

By taking $h=\varepsilon\tilde{h}$, we have the following corollary.
\begin{coro}
Under the assumption of Theorem~\ref{Convergence2}, we can get
\begin{equation}\label{errf}
\begin{aligned}
e^{n}(t)\le C(r^k+\frac{\beta^{s}}{r^{s-1}}+h_x^k+\frac{h}{\varepsilon}),
\end{aligned}
\end{equation}
where $e^{n}(t)=\max\limits_j(\vert X_j(t)-X_j^n(t)\vert+\vert V_{j}(t)-V_{j}^n(t)\vert)$, and $X_j^n(t),V_j^n(t)$ is the numerical solution.
\end{coro}
Actually, the more stronger the magnetic field is, the more "stable" the motion will be. This help the image $A_j(t)$ in Lemma~\ref{lemma2} do not become too "narrow". But if $\varepsilon$ is too small, it will restrict greatly the choice of temporal grid.

Note that the scheme (\ref{SCPD}) is a Hamiltonian splitting algorithm only if the magnetic field is uniform.
For the case of a nonhomogeneous magnetic field, since $e_w^n$ generally loses a factor on $\varepsilon$, we cannot introduce conclusions independent of $\varepsilon$ without imposing a limit on the step size.
Specifically, if we take the Hamiltonian splitting in a nonhomogeneous magnetic field to be of the form
\[
\Phi(\tilde{h})=\phi^{He}(\tilde{h})\circ\phi^{H{v}_z}(\tilde{h})\circ\phi^{H{v}_y}(\tilde{h})\circ\phi^{H{v}_x}(\tilde{h}),
\]
that is,
\begin{equation}\label{HSBX}
\begin{aligned}
&z^{n+1}_1=z^n_1+\varepsilon\tilde{h}w^n_1,\\
&z^{n+1}_2=z^n_2+\varepsilon\tilde{h}\left(w^n_2-{\bf{IB}}_{31}\right),\\
&z^{n+1}_3=z^n_3+\varepsilon\tilde{h}\left(w^n_3+{\bf{IB}}_{21}-{\bf{IB}}_{12}\right),\\
&w^{n+1}_1=w^n_1+{\bf{IB}}_{32}-{\bf{IB}}_{23}+\varepsilon\tilde{h}\tilde{E}_1(z^{n+1}),\\
&w^{n+1}_2=w^n_2+{\bf{IB}}_{13}-{\bf{IB}}_{31}+\varepsilon\tilde{h}\tilde{E}_2(z^{n+1}),\\
&w^{n+1}_3=w^n_3+{\bf{IB}}_{21}-{\bf{IB}}_{12}+\varepsilon\tilde{h}\tilde{E}_3(z^{n+1}),
\end{aligned}
\end{equation}
where $${\bf{IB}}_{ij}=\frac{1}{\varepsilon}\int_{z_j^n}^{z_j^{n+1}} B_i(\varepsilon z)dz_j.$$
And in integrating with respect to $z_1,z_2,z_3$, $z$ takes respectively
$$(z_1,z_2^n,z_3^n),\ (z_1^{n+1},z_2,z_3^n),\ (z_1^{n+1},z_2^{n+1},z_3).$$
By introducing the error terms $\theta_z$ and $\theta_w$, we have
\begin{align*}
z^{n+1}&=z^n+\varepsilon\int_0^{\tilde{h}} e^{s\hat{B}(\varepsilon z^n)}ds\,w^n-\theta_z,\\
w^{n+1}&=e^{\tilde{h}\hat{B}(\varepsilon z^n)}w^n+\varepsilon \tilde{h}\tilde{E}(z^{n+1})-\theta_w.
\end{align*}

Noting that
\begin{equation*}
\varepsilon\int_0^{\tilde{h}} \left(e^{s\hat{B}(\varepsilon z^n)}-I\right)ds=\varepsilon\int_0^{\tilde{h}}\int_0^{s}\hat{B}(\varepsilon z^n)e^{\sigma\hat{B}(\varepsilon z^n)}d\sigma ds,
\end{equation*}
and $\vert{\bf{IB}}_{ij}\vert\le C\tilde{h}$, we have $\vert\theta_z\vert\le C\varepsilon\tilde{h}^2$.

On the other hand, by Taylor expansion, there exist certain values $z^n$ and $z^{n+1}$ between $z_{\theta}$ such that
\begin{equation}\label{HSPNB}
\begin{aligned}
\theta_w&=\left(e^{\tilde{h}\hat{B}(\varepsilon z^n)}-I\right)w^n-\frac{1}{\varepsilon}\hat{B}(\varepsilon z_{\theta})(z^{n+1}-z^{n})\\
&=\int_0^{\tilde{h}}\hat{B}(\varepsilon z^n)e^{s\hat{B}(\varepsilon z^n)}ds\,w^n-\frac{1}{\varepsilon}\hat{B}(\varepsilon z_{\theta})(z^{n+1}-z^{n})\\
&=\frac{1}{\varepsilon}\hat{B}(\varepsilon z^n)(z^{n+1}-z^{n}+\theta_z)-\frac{1}{\varepsilon}\hat{B}(\varepsilon z_{\theta})(z^{n+1}-z^{n}).
\end{aligned}
\end{equation}

Thus $\vert\theta_w\vert\le C\tilde{h}^2$.
And these two error terms will appear in Eq. (\ref{errneed}), and since $\theta_w$ has no factor on $\varepsilon$, we can't deduce a result that agrees with Theorem \ref{lemmaCPD}, but the following corollary can be obtained.

\begin{coro}\label{HSBXC}
If $E^h\in C^{1}(\mathbb{R}^{3})$ and $B\in W^{1,\infty}(\mathbb{R}^{3})$, for scheme (\ref{HSBX}) we have
\begin{equation*}
\begin{aligned}
&\vert z^n-z(\tau_n)\vert+\vert w^n-w(\tau_n)\vert\le C(\tilde{h}+\frac{\tilde{h}^2}{\varepsilon}+h_x^k),\quad 0\le n\le\frac{T}{\varepsilon\tilde{h}},
\end{aligned}
\end{equation*}
for any $\tilde{h}\le \tilde{h}_0$ where $\tilde{h}_0$ is a constant independent of $\varepsilon$.
\end{coro}
By corollary \ref{HSBXC} we can find for a fixed parameter $\varepsilon$ if $\tilde{h}$ is much larger than $O(\varepsilon)$ then the scheme has quadratic convergence.
Once $\tilde{h}=O(\varepsilon)$ or less, the method changes back to first order.
Of course, this can be observed numerically. 
We use the following electromagnetic field
\begin{equation*}
E(x)=-x,\ B(x)=\frac{1}{\varepsilon}(0,0,1)^T+(x_1(x_3-x_2),x_2(x_1-x_3),x_3(x_2-x_1))^T.
\end{equation*}
The initial values are taken as $x(0)=(0.3,0.2,-1.4)^T$ and $v(0)=(-0.7,0.08,0.2)^T$, the termination time is $\frac{1}{\varepsilon}$.
The convergence order results are calculated as follows for different step sizes and parameters $\varepsilon$.

\begin{table}[!htbp]
    \centering
    \footnotesize
    \setlength{\tabcolsep}{4pt}
    \renewcommand{\arraystretch}{1.2}
    \begin{tabular}{|c|c|c|}
        \hline
        Step &  Error & Order\\
        \hline
        1/2 & 0.0072 & \\
        1/4 & 0.0018 & 2\\
        1/8 & 7.3470e-4 & 1.2928\\
        1/16 & 3.6784e-4 & 0.9981\\
        1/32 & 1.8450e-4 & 0.9955\\
        \hline
    \end{tabular}
    \caption{Accuracy order of scheme (\ref{HSBX}) in time by $\varepsilon=0.01$.}
\end{table}
\begin{table}[!htbp]
    \centering
    \footnotesize
    \setlength{\tabcolsep}{4pt}
    \renewcommand{\arraystretch}{1.2}
    \begin{tabular}{|c|c|c|}
        \hline
        Step &  Error & Order\\
        \hline
        1/8 & 4.9406e-4 & \\
        1/16 & 1.3036e-4 & 1.9222\\
        1/32 & 3.6681e-5 & 1.8294\\
        1/64 & 1.1477e-5 & 1.6763\\
        1/128 & 4.7091e-6 & 1.2852\\
        1/256 & 2.4135e-6 & 0.9643\\
        \hline
    \end{tabular}
    \caption{Accuracy order of scheme (\ref{HSBX}) in time by $\varepsilon=0.001$.}
\end{table}

In the rest of this section, we talk about the refined error of scheme (\ref{SCPD}) by dividing the time interval.
\begin{theorem}\label{thCPD}
Assume that $E^h\in C^{1}(\mathbb{R}^{3})$ and $B\in W^{1,\infty}(\mathbb{R}^{3})$, then there exists a constant $N_0$ independent of $\varepsilon$ such that when the time step $\tilde{h}=\frac{T_0}{N}$ for all $N\ge N_0$, we have
\begin{equation*}
\vert z^n-z(\tau_n)\vert+\vert w_{\parallel}^n-w_{\parallel}(\tau_n)\vert\le C(\varepsilon\tilde{h}+h_x^k), 0\le n\le \frac{T}{\varepsilon\tilde{h}}.
\end{equation*}
\end{theorem}

\begin{proof}
With a fixed $T>0$, we can have $\frac{T}{\varepsilon}=T_0 M+\tau_r$ with $0\le\tau_r<T_0$.
Due to the fact that $\tau_r$ is a cumulation of the truncation error, we only need to consider the case $\tau_r=0$.
Then we update the notations by denoting time scale $$\tau_n^m=mT_0+n\tilde{h},\quad 0\le m<M,\quad 0\le n\le N$$ with the step $\tilde{h}=\frac{T_0}{N}$, and the numerical solution at moment $\tau_n^m$ as $z^m_n$ and $w^m_n$.
Then the error is $$e_z^{n,m}=z(\tau_n^m)-z_n^{m},\quad e_w^{n,m}:=w(\tau_n^m)-w_n^{m}.$$
It should be noted that $$e_z^{0,m+1}=e_z^{N,m},\quad e_w^{0,m+1}=e_w^{N,m}.$$
The equation (\ref{erroreq}) now reads
\begin{equation}\label{error1}
\begin{aligned}
&e_z^{n+1,m}=e_z^{n,m}+\varepsilon \int_0^{\tilde{h}} e^{s\hat{B}(\varepsilon z(\tau_n^m))}ds\,e_w^{n,m}+\xi_z^{n,m}+\eta_z^{n,m}+\zeta_z^{n,m}\\
&e_w^{n+1,m}=e^{\tilde{h}\hat{B}(\varepsilon z(\tau_n^m))}e_w^{n,m}+\xi_w^{n,m}+\eta_w^{n,m}+\zeta_w^{n,m}.
\end{aligned}
\end{equation}
Moreover, it has
\begin{align*}
&\frac{1}{\varepsilon}\vert e_z^{j,m}\vert-\frac{1}{\varepsilon}\vert e_z^{j-1,m}\vert\le C(\tilde{h}\vert e_w^{j,m}\vert+\varepsilon{\tilde{h}}^2\vert e_z^{j,m}\vert+\varepsilon{\tilde{h}}^2),\\
&\vert e_w^{j,m}\vert-\vert e_w^{j-1,m}\vert\le C\varepsilon\tilde{h}(\vert e_z^{j,m}\vert+\vert e_z^{j-1,m}\vert+\tilde{h}+h_x^k).
\end{align*}
Then by Gronwall's inequality, we have $$\frac{1}{\varepsilon}\vert e_z^{n,m}\vert+\vert e_w^{n,m}\vert\le C(\tilde{h}+h_x^k+\frac{1}{\varepsilon}\vert e_z^{0,m}\vert+\vert e_w^{0,m}\vert).$$
Due to Theorem~\ref{lemmaCPD}, we can obtain the estimate
\begin{equation}\label{r2}
\begin{aligned}
\vert e_z^{n,m}\vert\le C(\varepsilon\tilde{h}+\varepsilon h_x^k+\vert e_z^{0,m}\vert).
\end{aligned}
\end{equation}
By summing (\ref{error1}), we have
\begin{equation}\label{error2}
\begin{aligned}
e_z^{N,m}&=e_z^{0,m}+\varepsilon\sum\limits_{n=0}^{N-1}\int_0^{\tilde{h}} e^{s\hat{B}(\varepsilon z(\tau_n^m))}ds\,e_w^{n,m}+\sum\limits_{n=0}^{N-1}(\xi_z^{n,m}+\eta_z^{n,m}+\zeta_z^{n,m})\\
&=e_z^{0,m}+\varepsilon\sum\limits_{n=0}^{N-1}\int_0^{\tilde{h}} e^{s\hat{B}(0)}ds\,e_w^{n,m}+\sum\limits_{n=0}^{N-1}(\xi_z^{n,m}+\eta_z^{n,m}+\zeta_z^{n,m})+\delta_z^m.
\end{aligned}
\end{equation}
Thanks to the Theorem~\ref{lemmaCPD} and Lemma~\ref{mang}, we have $$\vert\delta_z^m\vert\le C\varepsilon^2{\tilde{h}}^2.$$
On the other hand, for $n\le N,m\le M$ we have
$$e_w^{n,m}=e^{\tilde{h}\hat{B}(0)}e_w^{n-1,m}+\xi_w^{n-1,m}+\eta_w^{n-1,m}+\zeta_w^{n-1,m}+\delta_w^{n-1,m}$$
where $\delta_w^{n-1,m}=(e^{\tilde{h}\hat{B}(\varepsilon z(\tau_{n-1}^m))}-e^{\tilde{h}\hat{B}(0)})e_w^{n-1,m}$. And it directly has $\vert\delta_w^{n-1,m}\vert\le C\varepsilon{\tilde{h}}^2$ due to Lemma~\ref{mang}.
By recursion, we can get
$$e_w^{n,m}=e^{n\tilde{h}\hat{B}(0)}e_w^{0,m}+\sum\limits_{j=0}^{n-1}e^{(n-1-j)\tilde{h}\hat{B}(0)}(\xi_w^{j,m}+\eta_w^{j,m}+\zeta_w^{j,m}+\delta_w^{j,m}).$$
With the above equation, (\ref{error2}) can be rewritten as
\begin{equation}\label{error3}
\begin{aligned}
e_z^{N,m}&=e_z^{0,m}+\varepsilon\sum\limits_{n=0}^{N-1}\int_0^{\tilde{h}} e^{s\hat{B}(0)}ds\,e^{n\tilde{h}\hat{B}(0)}e_w^{0,m}+\gamma^m
\end{aligned}
\end{equation}
where
\begin{align}
\label{gamma1}\gamma^m=&\varepsilon\sum\limits_{n=0}^{N-1}\int_0^{\tilde{h}} e^{s\hat{B}(0)}ds\sum\limits_{j=0}^{n-1}e^{(n-1-j)h\hat{B}(0)}(\xi_w^{j,m}+\eta_w^{j,m}+\zeta_w^{j,m}+\delta_w^{j,m})\\
\notag+&\sum\limits_{n=0}^{N-1}(\xi_z^{n,m}+\eta_z^{n,m}+\zeta_z^{n,m})+\delta_z^m
\end{align}
For (\ref{gamma1}) we have
\begin{align*}
&\left\vert\varepsilon\sum\limits_{n=0}^{N-1}\int_0^{\tilde{h}} e^{s\hat{B}(0)}ds\sum\limits_{j=0}^{n-1}e^{(n-1-j)\tilde{h}\hat{B}(0)}(\xi_w^{j,m}+\eta_w^{j,m}+\zeta_w^{j,m}+\delta_w^{j,m})\right\vert\\
\le&C(\varepsilon^2\tilde{h}+\varepsilon^2\tilde{h}\sum\limits_{n=0}^{N-1}(\vert e_z^{n,m}\vert+\vert e_z^{n+1,m}\vert)).
\end{align*}
Therefore, it concludes $$\vert \gamma^m\vert\le C(\varepsilon^2\tilde{h}+\varepsilon^2\tilde{h}\sum\limits_{n=0}^{N-1}(\vert e_z^{n,m}\vert+\vert e_z^{n+1,m}\vert)).$$
Then from (\ref{error3}) we can get
\begin{equation}\label{error4}
\begin{aligned}
&\vert e_z^{N,m}\vert-\vert e_z^{0,m}\vert\\
\le &C\left(\varepsilon \left\vert\int_0^{\tilde{h}} e^{s\hat{B}(0)}ds\int_0^{T_0} e^{s\hat{B}(0)}ds\frac{e_w^{0,m}}{\tilde{h}}\right\vert+\varepsilon^2\tilde{h}+\varepsilon^2\tilde{h}\sum\limits_{n=0}^{N-1}(\vert e_z^{n,m}\vert+\vert e_z^{n+1,m}\vert)\right).
\end{aligned}
\end{equation}
By Lemma~\ref{mang}, we can obtain $$\int_0^{T_0} e^{s\hat{B}(0)}ds\cdot e_w^{0,m}=T_0(\bar{B}_0\cdot e_w^{0,m})\bar{B}_0$$ where $\bar{B}_0=\frac{B_0}{b}$.
Thus, (\ref{error4}) becomes
\begin{equation*}
\begin{aligned}
\vert e_z^{N,m}\vert-\vert e_z^{0,m}\vert&\le C\left(\varepsilon\vert(\bar{B}_0\cdot e_w^{0,m})\bar{B}_0\vert+\varepsilon^2\tilde{h}+\varepsilon^2\tilde{h}\sum\limits_{n=0}^{N-1}(\vert e_z^{n,m}\vert+\vert e_z^{n+1,m}\vert)\right)\\
&\le C\left(\varepsilon\vert e_{w,\parallel}^{0,m}\vert+\varepsilon^2\tilde{h}+\varepsilon^2\tilde{h}\sum\limits_{n=0}^{N-1}(\vert e_z^{n,m}\vert+\vert e_z^{n+1,m}\vert)\right).
\end{aligned}
\end{equation*}
where 
$$e_{w,\parallel}^{n,m}:=(\bar{B}^{n,m}\cdot e_w^{n,m})\bar{B}^{n,m},\quad \bar{B}^{n,m}:=\frac{B(\varepsilon z(\tau_{n}^m))}{\vert B(\varepsilon z(\tau_{n}^m))\vert}.$$
Due to $e_z^{N,m}=e_z^{0,m+1}$, we have
\begin{equation}\label{r4}
\begin{aligned}
\vert e_z^{0,m+1}\vert-\vert e_z^{0,m}\vert\le C(\varepsilon\vert e_{w,\parallel}^{0,m}\vert+\varepsilon^2(\vert e_z^{0,m}\vert+\vert e_z^{0,m+1}\vert)+\varepsilon^2\tilde{h}).
\end{aligned}
\end{equation}
By taking the inner product on both sides of (\ref{error1}) with the unit vector
$\bar{B}^{n+1,m}$, we can get
$$\vert e_{w,\parallel}^{n+1,m}\vert\le \vert\bar{B}^{n+1,m}\cdot(e^{h\hat{B}(\varepsilon z(\tau_n^m))}e_w^{n,m})\vert+\vert\eta_w^{n,m}+\zeta_w^{n,m}\vert+\vert\xi_w^{n,m}\cdot\bar{B}^{n+1,m}\vert.$$
As a result of $\bar{B}^{n+1,m}=\bar{B}^{n,m}+o(\varepsilon^2\tilde{h})$, we have $$\bar{B}^{n+1,m}\cdot(e^{\tilde{h}\hat{B}(\varepsilon z(\tau_n^m))}e_w^{n,m})=e_{w,\parallel}^{n,m}+o(\varepsilon^2 {\tilde{h}}^2).$$
Furthermore, it concludes
$$\vert e_{w,\parallel}^{n+1,m}\vert-\vert e_{w,\parallel}^{n,m}\vert\le C(\varepsilon\tilde{h}(\vert e_z^{n+1,m}\vert+\vert e_z^{n,m}\vert)+\vert\xi_w^{n,m}\cdot\bar{B}^{n,m}\vert+\varepsilon^2{\tilde{h}}^2).$$
By noting that
$$\left(\varepsilon\int_0^{\tilde{h}} e^{(\tilde{h}-s)\hat{B}(\varepsilon z(\tau_n^m))}E(\tilde{z}^{n,m}(s))ds-\varepsilon\int_0^{\tilde{h}} E(\tilde{z}^{n,m}(s))ds\right)\cdot\bar{B}^{n,m}=0$$ and
\begin{equation*}
\begin{aligned}
&\varepsilon\int_0^{\tilde{h}} E^h(\tilde{z}^{n,m}(s))ds-\varepsilon\tilde{h}E^h(\tilde{z}^{n,m}(\tilde{h}))\\
=&\varepsilon\int_0^{\tilde{h}} \left(E^h(\tilde{z}^{n,m}(\tilde{h}))+\nabla{E^h}(\tilde{z}_{\xi}^{n,m})\cdot(\tilde{z}^{n,m}(s)-\tilde{z}^{n,m}(\tilde{h}))\right)ds-\varepsilon\tilde{h}E^h(\tilde{z}^{n,m}(\tilde{h})),
\end{aligned}
\end{equation*}
we have
\begin{equation}\label{r1}
\begin{aligned}
\vert e_{w,\parallel}^{n+1,m}\vert-\vert e_{w,\parallel}^{n,m}\vert\le C(\varepsilon\tilde{h}(\vert e_z^{n+1,m}\vert+\vert e_z^{n,m}\vert)+\varepsilon^2{\tilde{h}}^2).
\end{aligned}
\end{equation}
Summing up (\ref{r1}) for $n$ from $0$ to $N-1$, it has
\begin{equation*}
\begin{aligned}
\vert e_{w,\parallel}^{0,m+1}\vert-\vert e_{w,\parallel}^{0,m}\vert\le C\left(\varepsilon\tilde{h}\sum\limits_{n=0}^{N-1}(\vert e_z^{n+1,m}\vert+\vert e_z^{n,m}\vert)+\varepsilon^2\tilde{h}\right).
\end{aligned}
\end{equation*}
Together with (\ref{r2}), we can obtain
\begin{equation}\label{r3}
\begin{aligned}
\vert e_{w,\parallel}^{0,m+1}\vert-\vert e_{w,\parallel}^{0,m}\vert\le C(\varepsilon(\vert e_z^{0,m}\vert+\vert e_z^{0,m+1}\vert)+\varepsilon^2\tilde{h}+\varepsilon h_x^k).
\end{aligned}
\end{equation}
Combining (\ref{r4}) and (\ref{r3}), we have
\begin{equation*}
\begin{aligned}
\vert e_{z}^{0,m+1}\vert+\vert e_{w,\parallel}^{0,m+1}\vert-\vert e_{z}^{0,m}\vert-\vert e_{w,\parallel}^{0,m}\vert\le C(\varepsilon(\vert e_z^{0,m}\vert+\vert e_{z}^{0,m}\vert+\vert e_z^{0,m+1}\vert)+\varepsilon^2\tilde{h}+\varepsilon h_x^k).
\end{aligned}
\end{equation*}
By Gronwall's inequality and $e_z^{0,0}=e_{w,\parallel}^{0,0}=0$, we have
\begin{equation*}
\begin{aligned}
\vert e_{z}^{0,m}\vert+\vert e_{w,\parallel}^{0,m}\vert\le C(\varepsilon\tilde{h}+h_x^k).
\end{aligned}
\end{equation*}
For $e_{z}^{n,m}$ and $e_{w,\parallel}^{n,m}$, it can be directly got from (\ref{r2}) and (\ref{r1}).
\end{proof}

For the Hamiltonian splitting method (\ref{HSBX}) of a non-uniform magnetic field, we can similarly show that
\begin{equation*}
\vert z^n-z(\tau_n)\vert+\vert w_{\parallel}^n-w_{\parallel}(\tau_n)\vert\le C(\varepsilon\tilde{h}+\tilde{h}^2+\varepsilon h_x^k),\quad 0\le n\le \frac{T}{\varepsilon\tilde{h}}.
\end{equation*}
It is sufficient to introduce an extra error $\varepsilon\tilde{h}^2$ in (\ref{r2}) due to the corollary \ref{HSBXC}, the proof of which will not be repeated.
On the other hand, by taking $h=\varepsilon\tilde{h}$ we have the following Corollary.
\begin{coro}\label{coroCPD}
Assume that $E^h\in C^{1}(\mathbb{R}^{3})$ and $B\in W^{1,\infty}(\mathbb{R}^{3})$, then there exists a constant $N_0$ independent of $\varepsilon$ such that when the time step $h=\varepsilon\frac{T_0}{N}$ for all $N\ge N_0$, we have
\begin{equation*}
\vert x^n-x(\tau_n)\vert+\vert v_{\parallel}^n-v_{\parallel}(\tau_n)\vert\le C(h+h_x^k), 0\le n\le \frac{T}{h}.
\end{equation*}
\end{coro}

\noindent{\bf Remark.} We can not have the refined error for $e_w^{n,m}$ because the corresponding $\xi_w^{n,m}$ only has a factor of $\varepsilon$.
However, $\vert\xi_w^{n,m}\cdot\bar{B}^{n,m}\vert=O(\varepsilon^2\tilde{h}^2)$ makes we can turn to consider the error $e_{w,\parallel}^{n,m}$.

\section{Conclusion}

In this paper, we have made a complete analysis for Hamiltonian particle methods of Vlasov--Poisson equations with a nonhomogeneous magnetic field whose strength is controlled by a parameter $\varepsilon$.
We use regularization functions to construct a mollified electric field and further develop a complete description and theory of the particle method.
Firstly, we show that the solution obtained by the particle method is a weak solution to the approximation problem of the original equation.
Based on the initial distribution we can generate a large number of initial macro particles whose trajectories satisfy characteristic equations of the Vlasov--Poisson equations.
The approximate solution of the system can be characterized by the weighted Dirac sum of these macro particles.
And after regularizing the electric field, the approximation problem corresponding to the particle method is equivalent to a zero-size particle model coupled with a smooth electric field, or a finite-size particle model.
By comparing the electric field and the mollified electric field we perform a convergence analysis of the particle method and prove the convergence of the approximate solution under a given class of nonhomogeneous magnetic fields. 
It is theoretically verified that the more particles are used for simulation and the smaller support of the regularization function is employed, the more accurate the solution is obtained by the particle method.

On the other hand, we consider the strong magnetic field with a small parameter $\varepsilon$ to control its strength.
By transforming the time scale, we transform the primal problem into a long time problem.
In this way we analyze the error of the fully discrete method and give the corresponding numerical convergence results.
If the applied magnetic field is homogeneous, we can give error results independent of $\varepsilon$ by truncating the system. 
When the applied magnetic field is nonhomogeneous, we likewise establish the corresponding error theory and verify it by numerical experiments. 
By further dividing the time due to the cyclotron period, with respect to the position and the parallel component of the velocity to the magnetic field we give optimal convergence results.

\backmatter
%
%
%
%
%
\bmhead{Acknowledgments}
%

This research was supported by the National Natural Science Foundation of China (12271513).

\bibliography{sn-bibliography}


\end{document}